\documentclass[12pt]{amsart}

\usepackage{accents}
\usepackage{appendix}
\usepackage{amsfonts}
\usepackage{amsmath}
\usepackage{amssymb}	
\usepackage{amsthm,bm}
\usepackage{array,booktabs,multirow}
\usepackage{braket}
\usepackage{centernot}
\usepackage{cite}
\usepackage{comment}
\usepackage{dsfont}
\usepackage{mathalfa}
\usepackage[shortlabels]{enumitem} 
\usepackage{etoolbox}
\usepackage{float}
\usepackage[hang, flushmargin]{footmisc}
\usepackage{latexsym}
\usepackage{lipsum}
\usepackage{needspace}
\usepackage{tikz}
\usepackage{hyperref}
\usetikzlibrary{matrix,arrows}  
\usepackage{kbordermatrix}

\renewcommand{\kbldelim}{(}
\renewcommand{\kbrdelim}{)}
\newcommand{\rvline}{\hspace*{-\arraycolsep}\vline\hspace*{-\arraycolsep}}

\theoremstyle{plain}
\newtheorem{thm}{Theorem}[section]
\newtheorem{cor}[thm]{Corollary}
\newtheorem{lem}[thm]{Lemma}
\newtheorem{prop}[thm]{Proposition}

\theoremstyle{definition}

\newtheorem{defn}[thm]{Definition}

\theoremstyle{remark}

\setlist[enumerate,1]{leftmargin=2em}

\def\H{\mathfrak H}

\def\F{\mathbb F}
\def\Z{\mathbb Z}

\def\U{U_q(\mathfrak{sl}_2)}
\def\e{\varepsilon}

\title[Finite-dimensional modules of $\triangle$ and $\H_q$]{Finite-dimensional modules of the universal Askey--Wilson algebra and DAHA of type $(C_1^\vee,C_1)$}

\author{Hau-Wen Huang}
\address{
Hau-Wen Huang\\
Department of Mathematics\\
National Central University\\
Chung-Li 32001 Taiwan
}
\email{hauwenh@math.ncu.edu.tw}

\begin{document}
\begin{abstract}
Assume that $\F$ is an algebraically closed field and let $q$ denote a nonzero scalar in $\F$ that is not a root of unity. The universal Askey--Wilson algebra $\triangle_q$ is a unital associative $\F$-algebra defined by generators and relations. The generators are $A,B, C$ and the relations state that each of 
\begin{gather*}
A+\frac{q BC-q^{-1} CB}{q^2-q^{-2}},
\qquad 
B+\frac{q CA-q^{-1} AC}{q^2-q^{-2}},
\qquad 
C+\frac{q AB-q^{-1} BA}{q^2-q^{-2}}
\end{gather*}
is central in $\triangle_q$. The universal DAHA (double affine Hecke algebra) $\mathfrak H_q$ of type $(C_1^\vee,C_1)$ is a unital associative $\F$-algebra generated by $\{t_i^{\pm 1}\}_{i=0}^3$ and the relations state that 
\begin{gather*}
t_it_i^{-1}=t_i^{-1} t_i=1
\quad 
\hbox{for all $i=0,1,2,3$};
\\
\hbox{$t_i+t_i^{-1}$ is central} 
\quad 
\hbox{for all $i=0,1,2,3$};
\\
t_0t_1t_2t_3=q^{-1}.
\end{gather*}
Each $\H_q$-module is a $\triangle_q$-module by pulling back via the injection $\triangle_q\to \H_q$ given by 
\begin{eqnarray*}
A &\mapsto & t_1 t_0+(t_1 t_0)^{-1},
\\
B &\mapsto & t_3 t_0+(t_3 t_0)^{-1},
\\
C &\mapsto & t_2 t_0+(t_2 t_0)^{-1}.
\end{eqnarray*}
We classify the lattices of $\triangle_q$-submodules of finite-dimensional irreducible $\H_q$-modules. As a corollary, for any finite-dimensional irreducible $\H_q$-module $V$, the $\triangle_q$-module $V$ is completely reducible if and only if $t_0$ is diagonalizable on $V$.
\end{abstract}

\maketitle



\allowdisplaybreaks

\section{Introduction}\label{s:introduction}

Throughout this paper, we adopt the following conventions. Assume that $\F$ is an algebraically closed field and fix a nonzero scalar $q\in \F$ that is not a root of unity.

In \cite{hidden_sym}, Zhedanov proposed the Askey--Wilson algebras to link the Askey--Wilson polynomials and the representation theory. In \cite{uaw2011}, Terwilliger introduced a central extension of the Askey--Wilson algebras, called the  universal Askey--Wilson algebra and denoted by $\triangle_q$. The algebra $\triangle_q$ is a unital associative $\F$-algebra defined by generators and relations. The generators are $A,B,C$. The relations assert that each of  
$$
A+\frac{q BC-q^{-1} CB}{q^2-q^{-2}},
\qquad 
B+\frac{q CA-q^{-1} AC}{q^2-q^{-2}},
\qquad 
C+\frac{q AB-q^{-1} BA}{q^2-q^{-2}}
$$ 
is central in $\triangle_q$. The algebra $\triangle_q$ has connections to the Racah--Wigner coefficients \cite{Huang:CG,Huang:RW}, 
the quantum algebra $\U$ \cite{uaw&equit2011}, 
 the $q$-Higgs algebra \cite{Higg&AW2019}, 
the $q$-Onsager algebra \cite{ter2018,uaw2011}, the $Q$-polynomial distance-regular graphs of $q$-Racah type \cite{qRacahDRG} and so on.

The Askey--Wilson polynomials can be obtained from the DAHA (double affine Hecke algebra) of type $(C_1^\vee,C_1)$. See for example \cite{koo07,koo08,Noumi04,Cherednik:2016,DAHA&OP:2019}. 
Inspired by the result, Terwilliger  \cite{DAHA2013} related the algebra $\triangle_q$ to the universal DAHA $\H_q$ of type $(C_1^\vee,C_1)$ in the following way: By definition the algebra $\H_q$ is a unital associative $\F$-algebra with generators $\{t_i^{\pm 1}\}_{i=0}^3$ and the relations state that 
\begin{gather*}
t_it_i^{-1}=t_i^{-1} t_i=1
\quad 
\hbox{for all $i=0,1,2,3$};
\\
\hbox{$t_i+t_i^{-1}$ is central} 
\quad 
\hbox{for all $i=0,1,2,3$};
\\
t_0t_1t_2t_3=q^{-1}.
\end{gather*}
The algebra $\H_q$ is a central extension of the DAHA of type $(C_1^\vee,C_1)$ \cite[\S 6.4]{DAHA&OP_book}. 
In \cite[Theorem 4.1]{DAHA2013} it was given an $\F$-algebra homomorphism $\zeta:\triangle_q\to \H_q$ that sends 
\begin{eqnarray*}
A &\mapsto & t_1 t_0+(t_1 t_0)^{-1},
\\
B &\mapsto & t_3 t_0+(t_3 t_0)^{-1},
\\
C &\mapsto & t_2 t_0+(t_2 t_0)^{-1}.
\end{eqnarray*}
Thus each $\H_q$-module is a $\triangle_q$-module by pulling back via $\zeta$. When $\F$ is the complex number field, the finite-dimensional irreducible modules over DAHA of type $(C_1^\vee,C_1)$ were first described in \cite{Oblomkov2009}.
In \cite{Huang:DAHAmodule,Huang:2015} the present author classified the finite-dimensional irreducible $\H_q$-modules and $\triangle_q$-modules, respectively. In this paper 
we are concerned with the lattices of $\triangle_q$-submodules of finite-dimensional irreducible $\H_q$-modules. The paper is structured as follows. In \S \ref{s:zeta} we recall some results concerning the homomorphism $\zeta:\triangle_q\to \H_q$. In \S\ref{s:AW&Hmodule}  we recall some facts concerning the finite-dimensional irreducible $\triangle_q$-modules and $\H_q$-modules. In \S\ref{s:lattice} we classify the lattices of $\triangle_q$-submodules of finite-dimensional irreducible $\H_q$-modules. In \S\ref{s:summary} we integrate the results of \S\ref{s:lattice} into a brief summary.

We mention some related works. 
Let $V$ denote a finite-dimensional irreducible $\H_q$-module. In \cite{daha&LP} Nomura and Terwilliger gave the sufficient conditions for $A$ and $B$ acting as the so-called Leonard pairs \cite{lp2001} on the eigenspaces of $t_0$ in $V$. In \cite{Huang:DAHA&LT} it is given the necessary and sufficient conditions for $A,B,C$ acting as Leonard triples \cite{cur2007,Huang:2012} on all composition factors of the $\triangle_q$-module $V$. 
The algebras $\triangle_q$ and $\H_q$ are the $q$-analogues of the universal Racah algebra $\Re$ and the universal additive DAHA $\H$ of type $(C_1^\vee,C_1)$, respectively \cite{Huang:BImodule,BI&NW2016,Huang:R<BI,SH:2019-1,Groenevelt2007}.
In \cite{Huang:R<BImodules} it is given the classification of the lattices of $\Re$-submodules of finite-dimensional irreducible $\H$-modules, provided that $\F$ is of characteristic zero. 
In  \cite{JaeHo2013} J.-H. Lee studied a relationship between the $\H_q$-modules and the $Q$-polynomial distance-regular graphs of $q$-Racah type.

\section{The universal Askey--Wilson algebra and DAHA of type $(C_1^\vee,C_1)$}\label{s:zeta}

In this section we more formally introduce the universal Askey--Wilson algebra $\triangle_q$ and the universal DAHA $\H_q$ of type $(C_1^\vee,C_1)$. Additionally we describe the homomorphism of $\triangle_q$ into $\H_q$ in further detail.

\begin{defn}
[Definition 2.1, \cite{uaw2011}]
\label{defn:UAW}
The  {\it universal Askey--Wilson algebra} $\triangle_q$ is a unital associative $\F$-algebra defined by generators and relations in the following way. The generators are $A, B, C$ and the relations state that each of
\begin{gather*}
A+\frac{q BC-q^{-1} CB}{q^2-q^{-2}},
\qquad 
B+\frac{q CA-q^{-1} AC}{q^2-q^{-2}},
\qquad 
C+\frac{q AB-q^{-1} BA}{q^2-q^{-2}}
\end{gather*}
commutes with $A,B,C$.  
\end{defn}

Define the three elements $\alpha,\beta,\gamma$ of $\triangle_q$ as follows:
\begin{align*}
\frac{\alpha}{q+q^{-1}}&=
A+\frac{q BC-q^{-1} CB}{q^2-q^{-2}},
\\
\frac{\beta}{q+q^{-1}}&=
B+\frac{q CA-q^{-1} AC}{q^2-q^{-2}},
\\
\frac{\gamma}{q+q^{-1}}&=
C+\frac{q AB-q^{-1} BA}{q^2-q^{-2}}.
\end{align*}
Note that $\alpha,\beta,\gamma$ are central in $\triangle_q$. Additionally, the remarkable element 
$$
\Omega=q ABC + q^2A^2 + q^{-2}B^2 + q^2C^2 
-qA\alpha - q^{-1}B\beta - qC\gamma
$$
is central in $\triangle_q$. We call $\Omega$ the {\it Casimir element} of $\triangle_q$ \cite{uaw2011,hidden_sym}.

\begin{lem}\label{lem:gen_UAW}
The algebra $\triangle_q$ is generated by $A,B,\gamma$.
\end{lem}
\begin{proof}
By Definition \ref{defn:UAW} the algebra $\triangle_q$ is generated by $A,B,C$.
By the setting of $\gamma$ we have
$$
C=\frac{\gamma}{q+q^{-1}}-\frac{qAB-q^{-1}BA}{q^2-q^{-2}}.
$$
Hence the lemma follows.
\end{proof}

The DAHA was introduced by Cherednik in connection with  Macdonald eigenvalue problems \cite{Cherednik:1995,Cherednik:1992}.  The DAHA of type $(C_1^\vee,C_1)$ is the most general DAHA of rank $1$. In  \cite{DAHA2013} a central extension of this algebra was proposed as follows:

\begin{defn}
[Definition 3.1, \cite{DAHA2013}] 
\label{defn:H}
The {\it universal DAHA $\H_q$ of type $(C_1^\vee,C_1)$} is a unital associative $\F$-algebra defined by generators and relations. The generators are $\{t_i^{\pm 1}\}_{i=0}^3$ and the relations state that  
\begin{gather}
t_it_i^{-1}=t_i^{-1} t_i=1
\quad 
\hbox{for all $i=0,1,2,3$};
\nonumber
\\
\hbox{$t_i+t_i^{-1}$ is central} 
\quad 
\hbox{for all $i=0,1,2,3$};
\label{ti+tiinv}
\\
t_0t_1t_2t_3=q^{-1}.
\label{t0123}
\end{gather}
\end{defn}

Define the elements $\{c_i\}_{i=0}^3$ and $X,Y$ of $\H_q$ as follows:
\begin{align}
c_i&=t_i+t_i^{-1}
\qquad 
\hbox{for all $i=0,1,2,3$;}
\label{ci}
\\
X&=t_3 t_0,
\label{X}
\\
Y&= t_0 t_1.
\label{Y}
\end{align}
Note that $c_i$ is central in $\H_q$.

\begin{thm}
[\S 4, \cite{DAHA2013}]
\label{thm:hom}
There exists a unique $\F$-algebra homomorphism $\zeta:\triangle_q\to \H_q$ that sends 
\begin{eqnarray*}
A &\mapsto & t_1 t_0+(t_1 t_0)^{-1},
\\
B &\mapsto & t_3 t_0+(t_3 t_0)^{-1},
\\
C &\mapsto & t_2 t_0+(t_2 t_0)^{-1},
\\
\alpha &\mapsto & 
c_3 c_2
+c_1
(q t_0^{-1}+q^{-1}t_0),
\\
\beta &\mapsto & 
c_2 c_1
+c_3
(q t_0^{-1}+q^{-1}t_0),
\\
\gamma &\mapsto & 
c_1 c_3
+c_2
(q t_0^{-1}+q^{-1}t_0).
\end{eqnarray*}
Moreover the image of $\Omega$ under $\zeta$ is equal to 
$$
(q+q^{-1})^2-(q^{-1}t_0 + q t_0^{-1})^2-c_1^2-c_2^2-c_3^2
-c_1 c_2 c_3(q^{-1}t_0 + q t_0^{-1}).
$$
\end{thm}

From now on each $\H_q$-module is viewed as a $\triangle_q$-module by pulling back via $\zeta$.

\section{Finite-dimensional irreducible $\triangle_q$-modules and $\H_q$-modules}\label{s:AW&Hmodule}

In \S\ref{s:AWmodule} we rephrase some results on the finite-dimensional irreducible $\triangle_q$-modules from \cite{Huang:2015}. In \S\ref{s:Hmodule_even} and \S\ref{s:Hmodule_odd} we recall from \cite{Huang:DAHAmodule} some results on the even-dimensional and odd-dimensional irreducible $\H_q$-modules, respectively.

\subsection{Finite-dimensional irreducible $\triangle_q$-modules}\label{s:AWmodule}

\begin{prop}
[\S 4.1, \cite{Huang:2015}]
\label{prop:UAWd}
For any nonzero scalars $a,b,c\in \F$ and any integer $d\geq 0$, there exists a $(d+1)$-dimensional $\triangle_q$-module $V_d(a,b,c)$  satisfying the following conditions {\rm (i)}, {\rm (ii)}: 
\begin{enumerate}
\item There exists an $\F$-basis for $V_d(a,b,c)$ with respect to which the matrices representing $A$ and $B$ are 
$$
\begin{pmatrix}
\theta_0 & & &  &{\bf 0}
\\ 
1 &\theta_1 
\\
&1 &\theta_2
 \\
& &\ddots &\ddots
 \\
{\bf 0} & & &1 &\theta_d
\end{pmatrix},
\qquad 
\begin{pmatrix}
\theta_0^* &\varphi_1 &  & &{\bf 0}
\\ 
 &\theta_1^* &\varphi_2
\\
 &  &\theta_2^* &\ddots
 \\
 & & &\ddots &\varphi_d
 \\
{\bf 0}  & & & &\theta_d^*
\end{pmatrix},
$$
respectively, where 
\begin{align*}
\theta_i 
&=
a q^{2i-d}+a^{-1} q^{d-2i}
\qquad 
\hbox{for $i=0,1,\ldots,d$},
\\
\theta^*_i 
&=
b q^{2i-d}+b^{-1} q^{d-2i}
\qquad 
\hbox{for $i=0,1,\ldots,d$},
\\
\varphi_i 
&=
a^{-1} b^{-1} q^{d+1}
(q^i-q^{-i})
(q^{i-d-1}-q^{d-i+1})
\\
&\qquad \times \;
(q^{-i}-a b c q^{i-d-1})
(q^{-i}-a b c^{-1} q^{i-d-1})
\qquad 
\hbox{for $i=1,2,\ldots,d$}.
\end{align*}

\item The elements $\alpha,\beta,\gamma$ act on $V_d (a,b,c)$ as scalar multiplication by
\begin{align*}
(b+b^{-1}) (c+c^{-1})+(a+a^{-1}) (q^{d+1}+q^{-d-1}),
\\
(c+c^{-1}) (a+a^{-1})+(b+b^{-1}) (q^{d+1}+q^{-d-1}),
\\
(a+a^{-1}) (b+b^{-1})+(c+c^{-1}) (q^{d+1}+q^{-d-1}),
\end{align*}
respectively.
\end{enumerate} 
\end{prop}

Note that the $\triangle_q$-module $V_d(a,b,c)$ from Proposition \ref{prop:UAWd} is unique up to isomorphism by Lemma \ref{lem:gen_UAW}.

\begin{thm}
[Theorem 4.4, \cite{Huang:2015}]
\label{thm:irr_UAW}
For any nonzero scalars $a,b,c\in \F$ and any integer $d\geq 0$, the $\triangle_q$-module $V_d(a,b,c)$ is irreducible if and only if $$
abc, 
a^{-1}bc, 
ab^{-1}c, 
abc^{-1}
\not\in 
\{q^{2i-d-1}\,|\, i=1,2,\ldots,d\}.
$$
\end{thm}

Define 
$$
[n]_q=\frac{q^n-q^{-n}}{q-q^{-1}}
\qquad 
\hbox{for all integers $n\geq 0$}.
$$

\begin{thm}
[Corollary 4.10, \cite{Huang:2015}]
\label{thm:onto_UAW}
Let $d\geq 0$ denote an integer. Suppose that $V$ is a $(d+1)$-dimensional irreducible $\triangle_q$-module. Let ${\rm tr}A,{\rm tr}B,{\rm tr}C$ denote the traces of $A,B,C$ on $V$, respectively. 
For any nonzero scalars $a,b,c\in \F$ the following are equivalent:
\begin{enumerate}
\item The $\triangle_q$-module $V_d(a,b,c)$ is isomorphic to $V$.

\item $a,b,c$ are the roots of the quadratic polynomials
\begin{align*}
&[d + 1]_q x^2 - {\rm tr}A x + [d + 1]_q, 
\\
&[d + 1]_q x^2 - {\rm tr}B x + [d + 1]_q, 
\\
&[d + 1]_q x^2 - {\rm tr}C x + [d + 1]_q, 
\end{align*}
respectively.
\end{enumerate}
\end{thm}

Let $\Z$ denote the additive group of integers.
Recall that $\Z/2\Z$ is the additive group of integers modulo $2$. Observe that there exists a unique $\Z/2\Z$-action on $\triangle_q$ such that each element of $\Z/2\Z$ acts on $\triangle_q$ as an $\F$-algebra automorphism in the following way:

\begin{table}[H]
\centering
\extrarowheight=3pt
\begin{tabular}{c|ccc|ccc}
$\e\in \Z/2\Z$  
&$A$ &$B$ &$C$ 
&$\alpha$ &$\beta$ &$\gamma$
\\

\midrule[1pt]

$0 \pmod{2}$ 
&$A$ &$B$ &$C$ 
&$\alpha$ &$\beta$ &$\gamma$
\\
$1 \pmod{2}$
&$B$ &$A$ &$C+\frac{AB-BA}{q-q^{-1}}$ 
&$\beta$ &$\alpha$ &$\gamma$ 
\end{tabular}
\caption{The $\Z/2\Z$-action on $\triangle_q$}\label{Z/2Z-action}
\end{table}

Let $V$ denote a $\triangle_q$-module. For any $\F$-algebra automorphism $\e$ of $\triangle_q$, the notation 
$$
V^\e
$$ stands for the $\triangle_q$-module obtained by twisting the $\triangle_q$-module $V$ via $\e$.

\begin{lem}\label{lem:Z/2Z}
Let $a,b,c$ denote nonzero scalars in $\F$ and let $d\geq 0$ be an integer. If the $\triangle_q$-module $V_d(a,b,c)$ is irreducible, then the $\triangle_q$-module $V_d(a,b,c)^{1\bmod{2}}$ is isomorphic to $V_d(b,a,c)$.
\end{lem}
\begin{proof}
By Table \ref{Z/2Z-action} the traces of $A,B,C$ on $V_d(a,b,c)^{1\bmod{2}}$ are equal to the traces of $B,A,C$ on $V_d(a,b,c)$, respectively. By Theorem \ref{thm:onto_UAW} the lemma follows.
\end{proof}

\subsection{Even-dimensional irreducible $\H_q$-modules}\label{s:Hmodule_even}

\begin{prop}
[Proposition 2.3, \cite{Huang:DAHAmodule}]
\label{prop:E}
Let $d\geq 1$ denote an odd integer. Assume that $k_0,k_1,k_2,k_3$ are nonzero scalars in $\F$ with 
$$
k_0^2=q^{-d-1}.
$$ 
Then there exists a $(d+1)$-dimensional $\H_q$-module $E(k_0,k_1,k_2,k_3)$ satisfying the following conditions:
\begin{enumerate}
\item There exists an $\F$-basis $\{v_i\}_{i=0}^d$ for $E(k_0,k_1,k_2,k_3)$ such that 
\begin{align*}
t_0 v_i 
&=
\left\{
\begin{array}{ll}
\textstyle
k_0^{-1} q^{-i} (1-q^i) (1-k_0^2 q^i) 
v_{i-1}
+
(
k_0+k_0^{-1}-k_0^{-1}q^{-i}
) 
v_i
\qquad
&\hbox{for $i=2,4,\ldots,d-1$},
\\
k_0^{-1} q^{-i-1}
(v_i-v_{i+1})
\qquad
&\hbox{for $i=1,3,\ldots,d-2$},
\end{array}
\right.
\\
t_0 v_0&=k_0 v_0,
\qquad
t_0 v_d=k_0v_d,
\\
t_1 v_i
&=
\left\{
\begin{array}{ll}
-k_1(1-q^i)(1-k_0^2 q^i)v_{i-1}
+k_1 v_i
+k_1^{-1} v_{i+1}
\qquad
&\hbox{for $i=2,4,\ldots,d-1$},
\\
k_1^{-1} v_i
\qquad
&\hbox{for $i=1,3,\ldots,d$},
\end{array}
\right.
\\
t_1 v_0 &=k_1 v_0+k_1^{-1} v_1,
\\
t_2 v_i 
&=
\left\{
\begin{array}{ll}
k_0^{-1} k_1^{-1} k_3^{-1} q^{-i-1}
(v_i-v_{i+1})
\qquad
&\hbox{for $i=0,2,\ldots,d-1$},
\\
\textstyle
\frac{(k_0 k_1 k_3 q^i-k_2)
(k_0 k_1 k_3 q^i- k_2^{-1})}
{k_0 k_1 k_3 q^i  }
v_{i-1}
+
(k_2+k_2^{-1}-
k_0^{-1} k_1^{-1} k_3^{-1} q^{-i}
) v_i
\qquad
&\hbox{for $i=1,3,\ldots,d$},
\end{array}
\right.
\\
t_3 v_i 
&=
\left\{
\begin{array}{ll}
k_3 v_i
\qquad
&\hbox{for $i=0,2,\ldots,d-1$},
\\
-k_3^{-1}(k_0k_1k_3q^i-k_2)
(k_0k_1k_3 q^i-k_2^{-1})
v_{i-1}
+k_3^{-1} v_i
+k_3 v_{i+1}
\qquad
&\hbox{for $i=1,3,\ldots,d-2$},
\end{array}
\right.
\\
t_3 v_d &=
-k_3^{-1}(k_0k_1k_3q^d-k_2)
(k_0k_1k_3 q^d-k_2^{-1})
v_{d-1}
+k_3^{-1} v_d.
\end{align*}

\item The elements $c_0, c_1,c_2,c_3$ act on $E(k_0,k_1,k_2,k_3)$ as scalar multiplication by 
$$
k_0+k_0^{-1},
\quad 
k_1+k_1^{-1},
\quad 
k_2+k_2^{-1},
\quad 
k_3+k_3^{-1},
$$
respectively.
\end{enumerate}
\end{prop}

Recall the elements $X$ and $Y$ of $\H_q$ from (\ref{X}) and (\ref{Y}).

\begin{lem}
\label{lem:XYinE} 
With reference to Proposition \ref{prop:E}, the following {\rm (i)} and {\rm (ii)} hold:
\begin{enumerate}
\item The action of $X$ on $E(k_0,k_1,k_2,k_3)$ is as follows:
\begin{align*}
(1-k_0 k_3 q^{2\lceil \frac{i}{2}\rceil}X^{(-1)^{i-1}}) v_i
=
\left\{
\begin{array}{ll}
0
\qquad 
&\hbox{if $i=0$},
\\
\varrho_i
v_{i-1}
\qquad 
&\hbox{if $i=1,2,\ldots,d$},
\end{array}
\right.
\end{align*}
where 
$$
\varrho_i
=\left\{
\begin{array}{ll}
(k_0 k_1 k_3 q^i-k_2)(k_0 k_1 k_3 q^i-k_2^{-1})
\qquad 
&\hbox{for $i=1,3,\ldots,d$},
\\
(1-q^i) (1-k_0^2 q^i)
\qquad 
&\hbox{for $i=2,4,\ldots,d-1$}.
\end{array}
\right.
$$

\item The action of $Y$ on $E(k_0,k_1,k_2,k_3)$ is as follows:
\begin{align*}
(1-k_0 k_1 q^{2\lceil \frac{i}{2}\rceil} Y^{(-1)^{i-1}}) v_i
=
\left\{
\begin{array}{ll}
0
\qquad 
&\hbox{if $i=d$},
\\
v_{i+1}
\qquad 
&\hbox{if $i=0,1,\ldots,d-1$}.
\end{array}
\right.
\end{align*}
\end{enumerate}
\end{lem}
\begin{proof}
Evaluate the actions of $X,Y$ on $E(k_0,k_1,k_2,k_3)$ by Proposition \ref{prop:E}.
\end{proof}

\begin{thm}
[Theorem 5.8, \cite{Huang:DAHAmodule}]
\label{thm:irr_E}
For any odd integer $d\geq 1$ and any nonzero $k_0,k_1,k_2,k_3\in \F$ with $k_0^2=q^{-d-1}$, the $\H_q$-module $E(k_0,k_1,k_2,k_3)$ is irreducible if and only if 
$$
k_0k_1k_2k_3,
k_0k_1^{-1}k_2k_3, 
k_0k_1k_2^{-1}k_3, 
k_0k_1k_2 k_3^{-1}
\not\in
\{q^{-i}\,|\, i=1,3,\ldots,d\}.
$$ 
\end{thm}

Recall that $\Z/4\Z$ is the additive group of integers modulo $4$. 
Observe that there exists a unique $\Z/4\Z$-action on $\H_q$ such that each element of $\Z/4\Z$ acts on $\H_q$ as an $\F$-algebra automorphism in the following way:

\begin{table}[H]
\centering
\extrarowheight=3pt
\begin{tabular}{c|rrrr}
$\e\in \Z/4\Z$  &$t_0$ &$t_1$ &$t_2$ &$t_3$ 
\\

\midrule[1pt]

${0\pmod 4}$ &$t_0$  &$t_1$ &$t_2$ &$t_3$ 
\\
${1\pmod 4}$ &$t_1$ &$t_2$ &$t_3$ &$t_0$ 
\\
${2\pmod 4}$ &$t_2$ &$t_3$ &$t_0$ &$t_1$
\\
${3\pmod 4}$ &$t_3$ &$t_0$ &$t_1$ &$t_2$
\end{tabular}
\caption{The $\Z/4\Z$-action on $\H_q$}\label{Z/4Z-action}
\end{table}

Let $V$ denote an $\H_q$-module. For any $\F$-algebra automorphism $\e$ of $\H_q$ the notation 
$$
V^\e
$$ 
stands for the $\H_q$-module obtained by twisting the $\H_q$-module $V$ via $\e$.

\begin{thm}
[Theorem 6.1, \cite{Huang:DAHAmodule}]
\label{thm:onto_E}
Let $d\geq 1$ denote an odd integer. If $V$ is a $(d+1)$-dimensional irreducible $\H_q$-module, then there exist an $\e\in \Z/4\Z$ and nonzero $k_0,k_1,k_2,k_3\in \F$ with $k_0^2=q^{-d-1}$ such that the $\H_q$-module $E(k_0,k_1,k_2,k_3)^\e$ is isomorphic to $V$.
\end{thm}

\subsection{Odd-dimensional irreducible $\H_q$-modules}\label{s:Hmodule_odd}

\begin{prop}
[Proposition 2.6, \cite{Huang:DAHAmodule}]
\label{prop:O}
Let $d\geq 0$ denote an even integer. Assume that $k_0,k_1,k_2,k_3$ are nonzero scalars in $\F$ with 
$$
k_0 k_1 k_2 k_3=q^{-d-1}.
$$
Then there exists a $(d+1)$-dimensional $\H_q$-module
 $O(k_0,k_1,k_2,k_3)$ satisfying the following conditions:
\begin{enumerate} 
\item There exists an $\F$-basis $\{v_i\}_{i=0}^d$ for $O(k_0,k_1,k_2,k_3)$  such that 
\begin{align*}
t_0 v_i 
&=
\left\{
\begin{array}{ll}
\textstyle
k_0^{-1} q^{-i} (1-q^i) (1-k_0^2 q^i) 
v_{i-1}
+
(
k_0+k_0^{-1}-k_0^{-1}q^{-i}
) 
v_i
\qquad
&\hbox{for $i=2,4,\ldots,d$},
\\
\textstyle
k_0^{-1} q^{-i-1}
(v_i-v_{i+1})
\qquad
&\hbox{for $i=1,3,\ldots,d-1$},
\end{array}
\right.
\\
t_0 v_0&= k_0 v_0,
\\
t_1 v_i
&=
\left\{
\begin{array}{ll}
-k_1(1-q^i)(1-k_0^2 q^i)v_{i-1}
+k_1 v_i
+k_1^{-1} v_{i+1}
\qquad
&\hbox{for $i=2,4,\ldots,d-2$},
\\
k_1^{-1} v_i
\qquad
&\hbox{for $i=1,3,\ldots,d-1$},
\end{array}
\right.
\\
t_1 v_0&= k_1 v_0 +k_1^{-1} v_1,
\qquad 
t_1 v_d=-k_1(1-q^d)(1-k_0^2 q^d)v_{d-1}
+k_1 v_d,
\\
t_2 v_i 
&=
\left\{
\begin{array}{ll}
k_2 q^{d-i}
(v_i-v_{i+1})
\qquad
&\hbox{for $i=0,2,\ldots,d-2$},
\\
\textstyle
-k_2(1-k_2^{-2}q^{i-d-1})
(1-q^{d-i+1})
v_{i-1}
+
(k_2+k_2^{-1}-
k_2 q^{d-i+1}) v_i
\qquad
&\hbox{for $i=1,3,\ldots,d-1$},
\end{array}
\right.
\\
t_2 v_d &= k_2 v_d,
\\
t_3 v_i 
&=
\left\{
\begin{array}{ll}
k_3 v_i
\qquad
&\hbox{for $i=0,2,\ldots,d$},
\\
\textstyle
-k_3^{-1}
(1-k_2^{-2} q^{i-d-1})
(1-q^{i-d-1})
v_{i-1}
+k_3^{-1} v_i
+k_3 v_{i+1}
\qquad
&\hbox{for $i=1,3,\ldots,d-1$}.
\end{array}
\right.
\end{align*}
\item The elements $c_0, c_1,c_2,c_3$ act on $O(k_0,k_1,k_2,k_3)$ as scalar multiplication by 
$$
k_0+k_0^{-1},
\quad 
k_1+k_1^{-1},
\quad 
k_2+k_2^{-1},
\quad 
k_3+k_3^{-1},
$$
respectively.
\end{enumerate}
\end{prop}

\begin{lem}
\label{lem:XYinO}
With reference to Proposition \ref{prop:O}, the following {\rm (i)} and {\rm (ii)} hold:
\begin{enumerate}
\item The action of $X$ on $O(k_0,k_1,k_2,k_3)$ is as follows:
\begin{align*}
(1-k_0 k_3 q^{2\lceil \frac{i}{2}\rceil}X^{(-1)^{i-1}}) v_i
=
\left\{
\begin{array}{ll}
0
\qquad 
&\hbox{if $i=0$},
\\
\varrho_i
v_{i-1}
\qquad 
&\hbox{if $i=1,2,\ldots,d$},
\end{array}
\right.
\end{align*}
where 
$$
\varrho_i
=\left\{
\begin{array}{ll}
(q^{i-d-1}-1)(k_2^{-2} q^{i-d-1}-1)
\qquad 
&\hbox{for $i=1,3,\ldots,d-1$},
\\
(1-q^i) (1-k_0^2 q^i)
\qquad 
&\hbox{for $i=2,4,\ldots,d$}.
\end{array}
\right.
$$

\item The action of $Y$ on $O(k_0,k_1,k_2,k_3)$ is as follows:
\begin{align*}
(1-k_0 k_1 q^{2\lceil \frac{i}{2}\rceil} Y^{(-1)^{i-1}}) v_i
=
\left\{
\begin{array}{ll}
0
\qquad 
&\hbox{if $i=d$},
\\
v_{i+1}
\qquad 
&\hbox{if $i=0,1,\ldots,d-1$}.
\end{array}
\right.
\end{align*}
\end{enumerate}
\end{lem}
\begin{proof}
Evaluate the actions of $X,Y$ on $O(k_0,k_1,k_2,k_3)$ by Proposition \ref{prop:O}.
\end{proof}

\begin{thm}
[Theorem 7.7, \cite{Huang:DAHAmodule}]
\label{thm:irr_O}
For any even integer $d\geq 0$ and any nonzero $k_0,k_1,k_2,k_3\in \F$ with $k_0 k_1 k_2 k_3=q^{-d-1}$, 
the $\H_q$-module $O(k_0,k_1,k_2,k_3)$ is irreducible if and only if  
$$
k_0^2,k_1^2,k_2^2,k_3^2\not\in
\{q^{-i}\,|\, i=2,4,\ldots,d\}.
$$
\end{thm}

\begin{thm}
[Theorem 8.1, \cite{Huang:DAHAmodule}]
\label{thm:onto_O}
Let $d\geq 0$ denote an even integer. If $V$ is a $(d+1)$-dimensional irreducible $\H_q$-module, then there exist nonzero $k_0,k_1,k_2,k_3\in \F$ with $k_0 k_1 k_2 k_3=q^{-d-1}$ such that the $\H_q$-module $O(k_0,k_1,k_2,k_3)$ is isomorphic to $V$.
\end{thm}

\section{The lattices of $\triangle_q$-submodules of finite-dimensional $\H_q$-modules}\label{s:lattice}

In \S\ref{s:t0} we investigate the role of $t_0$ in the $\triangle_q$-submodules of an $\H_q$-module. 
In \S\ref{s:lattice_E}--\ref{s:lattice_O} we inspect the $\triangle_q$-submodules of the irreducible $\H_q$-modules following the results of Theorems \ref{thm:onto_E} and \ref{thm:onto_O}.

\subsection{The eigenspaces of $t_0$ and $\triangle_q$-modules}\label{s:t0}

By \cite[Theorem 4.5]{DAHA2013}  the $\F$-algebra homomorphism $\zeta:\triangle_q\to \H_q$ given in Theorem \ref{thm:hom} is injective. Thus the universal Askey--Wilson algebra $\triangle_q$ can be considered as a subalgebra of $\H_q$.

Let $\mathcal A$ denote an algebra.
Recall that the commutator $[x,y]$ of two elements $x,y\in \mathcal A$ is defined by $[x,y]=xy-yx$. Given a subset $S$ of $\mathcal A$, the centralizer of $S$ in $\mathcal A$ is the set of all elements $x\in \mathcal A$ satisfying $[x,y]=0$ for all $y\in S$.

\begin{lem}\label{lem:t0_centralizer}
The element $t_0$ is in the centralizer of $\triangle_q$ in $\H_q$. 
\end{lem}
\begin{proof}
It suffices to show that $t_0$ commutes with each of $A,B,C$ by Definition \ref{defn:UAW}. Using $t_0^2=c_0 t_0-1$  yields that  
\begin{gather}\label{[t0A]1}
[t_0,t_1t_0]=t_1-t_0^{-1} t_1 t_0.
\end{gather}
Using $t_1^{-1}=c_1-t_1$ yields that 
\begin{gather}\label{[t0A]2}
[t_0,(t_1t_0)^{-1}]=t_0^{-1} t_1 t_0-t_1.
\end{gather}
Recall that $A=t_1t_0+(t_1t_0)^{-1}$ from Theorem \ref{thm:hom}. Adding both sides of (\ref{[t0A]1}) and (\ref{[t0A]2}) yields that 
$[t_0,A]=0$. By similar arguments each of $[t_0,B]$ and $[t_0,C]$ is zero. The lemma follows.
\end{proof}

\begin{lem}\label{lem:X+Xinv}
The following equations hold in $\H_q$:
\begin{enumerate}
\item $t_0t_3+(t_0t_3)^{-1}=t_3t_0+(t_3t_0)^{-1}=X+X^{-1}$.

\item $t_1t_0+(t_1t_0)^{-1}=t_0t_1+(t_0t_1)^{-1}=Y+Y^{-1}$.

\item $t_2t_1+(t_2t_1)^{-1}=t_1t_2+(t_1t_2)^{-1}=q X+q^{-1} X^{-1}$.

\item $t_3t_2+(t_3t_2)^{-1}=t_2t_3+(t_2t_3)^{-1}=q Y+q^{-1} Y^{-1}$.
\end{enumerate}
\end{lem}
\begin{proof}
Recall the $\Z/4\Z$-action on $\H_q$ from Table \ref{Z/4Z-action}. Set $\e=1\pmod{4}$. 

(i): Using Lemma \ref{lem:t0_centralizer} yields the first equality. By (\ref{X}) the second equality holds.

(ii): It follows by applying $\e$ to Lemma \ref{lem:X+Xinv}(i).

(iii): The first equality follows by applying $\e$ to Lemma \ref{lem:X+Xinv}(ii). By (\ref{t0123}) the element $t_1t_2=q^{-1} X^{-1}$. Hence the second equality holds. 

(iv): It follows by applying $\e$ to Lemma \ref{lem:X+Xinv}(iii).
\end{proof}

Given any $\H_q$-module $V$ and any $\theta\in \F$ we let 
$$
V(\theta)=\{v\in V\,|\, t_0v=\theta v\}.
$$

\begin{prop}\label{prop:t0eigenspace=AWmodule}
Let $V$ denote an $\H_q$-module. Then $V(\theta)$ is a $\triangle_q$-submodule of $V$ for any $\theta\in \F$.
\end{prop}
\begin{proof}
It follows from Lemma \ref{lem:t0_centralizer} that $V(\theta)$ is $x$-invariant for all $x\in \triangle_q$ for any $\theta\in \F$. 
\end{proof}

\begin{lem}\label{lem:irr_AWmodule_in_t0eigenspace}
If $V$ is a finite-dimensional irreducible $\H_q$-module, then at least one of $\{c_i\}_{i=0}^3$ acts on $V$ as multiplication by a nonzero scalar.
\end{lem}
\begin{proof}
Let $d\geq 0$ denote an integer with $\dim V=d+1$. We divide the argument into the two cases: (i) $d$ is odd; (ii) $d$ is even.

(i): By Theorem \ref{thm:onto_E} there are an $\e\in \Z/4\Z$ and nonzero $k_0,k_1,k_2,k_3\in \F$ with $k_0^2=q^{-d-1}$ such that $V$ is isomorphic to $E(k_0,k_1,k_2,k_3)^\e$.  Since $q$ is not a root of unity, the value $k_0^2$ is not equal to $-1$. By Proposition \ref{prop:E}(ii) the element $c_0$ acts on $E(k_0,k_1,k_2,k_3)$ as scalar multiplication by $k_0+k_0^{-1}\not=0$. Therefore the lemma holds for this case.

(ii): By Theorem \ref{thm:onto_O} there are nonzero $k_0,k_1,k_2,k_3\in \F$ with $k_0 k_1 k_2 k_3=q^{-d-1}$ such that $V$ is isomorphic to $O(k_0,k_1,k_2,k_3)$.  Since $q$ is not a root of unity, the value $k_0 k_1 k_2 k_3$ is not equal to $\pm 1$. 
Hence there exists an $i\in \{0,1,2,3\}$ such that $k_i^2\not= -1$.
By Proposition \ref{prop:O}(ii) the element $c_i$ acts on $O(k_0,k_1,k_2,k_3)$ as scalar multiplication by $k_i+k_i^{-1}\not=0$. Therefore the lemma holds for this case.
\end{proof}

\begin{prop}\label{prop:irr_AWmodule_in_t0eigenspace}
Let $V$ denote a finite-dimensional irreducible $\H_q$-module. For any irreducible $\triangle_q$-submodule $W$ of $V$, there exists a nonzero scalar $\theta\in \F$ such that $W\subseteq V(\theta)$.
\end{prop}
\begin{proof}
Suppose that $W$ is an irreducible $\triangle_q$-submodule of $V$.
By Schur's lemma we may divide the argument into the two cases: (i) At least one of $\{c_i\}_{i=1}^3$ acts on $V$ as multiplication by a nonzero scalar. (ii) Each of $\{c_i\}_{i=1}^3$ vanishes on $V$.

(i): First we assume that $c_1$ acts on $V$ as multiplication by a nonzero scalar. 
Observe that 
$$
q^{-1} t_0+q t_0^{-1}=q c_0-(q-q^{-1})t_0.
$$ 
Since $q^2\not=1$ it follows from Theorem \ref{thm:hom} that $t_0$ is an $\F$-linear combination of $1$ and $\alpha$ 
as endomorphisms of $V$. By Schur's lemma the element $\alpha$ acts on $W$ as scalar multiplication. Hence $t_0$ acts on $W$ as multiplication by a scalar $\theta\in \F$. Hence $W\subseteq V(\theta)$. Since $t_0$ is invertible in $\H_q$ the scalar $\theta$ is nonzero.

When $c_2$ or $c_3$ acts on $V$ as multiplication by a nonzero scalar, the proposition is true by a similar argument.

(ii): Observe that 
$$
(q^{-1} t_0+q t_0^{-1})^2=-(q-q^{-1})^2+q^2 c_0^2-c_0(q^2-q^{-2}) t_0.
$$ 
By Lemma \ref{lem:irr_AWmodule_in_t0eigenspace} the element $c_0$ acts on $V$ as multiplication by a nonzero scalar. Combined with $q^4\not=1$ this yields that $t_0$ is an $\F$-linear combination of $1$ and $\Omega$ 
as endomorphisms of $V$ by Theorem \ref{thm:hom}. By Schur's lemma the element $\Omega$ acts on $W$ as scalar multiplication. Hence $t_0$ acts on $W$ as multiplication by a scalar $\theta\in \F$.  Hence $W\subseteq V(\theta)$. Since $t_0$ is invertible in $\H_q$ the scalar $\theta$ is nonzero.  
\end{proof}

\subsection{The lattice of $\triangle_q$-submodules of $E(k_0,k_1,k_2,k_3)$}\label{s:lattice_E}

Throughout \S\ref{s:lattice_E}--\S\ref{s:lattice_E3} we use the following conventions: Let $d\geq 1$ denote an odd integer and assume that $k_0,k_1,k_2,k_3$ are nonzero scalars in $\F$ with 
\begin{gather}\label{k02=-d-1}
k_0^2=q^{-d-1}.
\end{gather}
Let $\{v_i\}_{i=0}^d$ denote the $\F$-basis for $E(k_0,k_1,k_2,k_3)$ from Proposition \ref{prop:E}(i). Set 
$$
\varrho_i
=\left\{
\begin{array}{ll}
(k_0 k_1 k_3 q^i-k_2)(k_0 k_1 k_3 q^i-k_2^{-1})
\qquad 
&\hbox{for $i=1,3,\ldots,d$},
\\
(1-q^i) (1-k_0^2 q^i)
\qquad 
&\hbox{for $i=2,4,\ldots,d-1$}.
\end{array}
\right.
$$

In this subsection, we study the $\triangle_q$-submodules of the $\H_q$-module $E(k_0,k_1,k_2,k_3)$. Recall that $A=t_1 t_0+(t_1t_0)^{-1}$ and $B=t_3t_0+(t_3t_0)^{-1}$ from Theorem \ref{thm:hom}.

\begin{lem}\label{lem:AB_E}
The actions of $A$ and $B$ on the $\H_q$-module $E(k_0,k_1,k_2,k_3)$ are as follows:
\begin{align*}
A v_i &=
\left\{
\begin{array}{ll}
\displaystyle
\theta_i v_i 
-
k_0^{-1} k_1^{-1} q^{-i-1}(q-q^{-1}) v_{i+1}
-
k_0^{-1} k_1^{-1} q^{-i-2} v_{i+2}
\qquad 
&\hbox{for $i=0,2,\ldots,d-3$},
\\
\displaystyle
\theta_i v_i -k_0^{-1} k_1^{-1} q^{-i-1} v_{i+2}
\qquad 
&\hbox{for $i=1,3,\ldots,d-2$},
\end{array}
\right.
\\
A v_{d-1}&=\theta_{d-1} v_{d-1} -k_0^{-1} k_1^{-1} q^{-d} (q-q^{-1}) v_d,
\qquad 
A v_d = \theta_d v_d,
\\
B v_i &=
\left\{
\begin{array}{ll}
\displaystyle
\theta_i^* v_i -k_0^{-1} k_3^{-1} q^{-i} \varrho_i \varrho_{i-1} v_{i-2}
\qquad 
&\hbox{for $i=2,4,\ldots,d-1$},
\\
\displaystyle
\theta_i^* v_i
+
k_0^{-1} k_3^{-1} q^{-i} (q-q^{-1}) \varrho_i  v_{i-1}
-
k_0^{-1} k_3^{-1} q^{1-i} \varrho_i \varrho_{i-1} v_{i-2}
\qquad 
&\hbox{for $i=3,5,\ldots,d$},
\end{array}
\right.
\\
B v_0 &=
\theta^*_0 v_0,
\qquad 
B v_1 =
\theta^*_1 v_1+
k_0^{-1} k_3^{-1} q^{-1} (q-q^{-1}) \varrho_1  v_0,
\end{align*}
where 
\begin{align*}
\theta_i
&=
k_0 k_1 q^{2\lceil \frac{i}{2}\rceil}
+k_0^{-1} k_1^{-1}  q^{-2\lceil \frac{i}{2}\rceil}
\qquad 
\hbox{for $i=0,1,\ldots,d$},
\\
\theta_i^*
&=
k_0 k_3 q^{2\lceil \frac{i}{2}\rceil}
+k_0^{-1} k_3^{-1}  q^{-2\lceil \frac{i}{2}\rceil}
\qquad 
\hbox{for $i=0,1,\ldots,d$}.
\end{align*}
\end{lem}
\begin{proof}
By Lemma \ref{lem:X+Xinv} we have $A=Y+Y^{-1}$ and $B=X+X^{-1}$. Using Lemma \ref{lem:XYinE} it is straightforward to verify the lemma.
\end{proof}

\begin{lem}\label{lem1:t0_E}
The matrix representing $t_0$ with respect to the $\F$-basis 
\begin{gather*}
v_0,
\quad 
v_d,
\quad 
v_i+(q^i-1)v_{i-1}
\quad \hbox{for $i=2,4,\ldots,d-1$},
\quad 
q^{i+1} v_i 
\quad \hbox{for $i=1,3,\ldots,d-2$}
\end{gather*}
for $E(k_0,k_1,k_2,k_3)$ is 
\begin{gather*}
\begin{pmatrix}
k_0 I_2  &\rvline &{\bf 0}  &\rvline &{\bf 0} 
\\
\hline
{\bf 0}  &\rvline
&k_0 I_{\frac{d-1}{2}}
 &\rvline & -k_0^{-1} I_{\frac{d-1}{2}}
\\
\hline
{\bf 0} &  \rvline 
&{\bf 0} &  \rvline
&k_0^{-1} I_{\frac{d-1}{2}} 
\end{pmatrix}.
\end{gather*}
\end{lem}
\begin{proof}
It is routine to verify the lemma by using Proposition \ref{prop:E}(i).
\end{proof}

\begin{lem}\label{lem2:t0_E}
\begin{enumerate}
\item If $d=1$ then $t_0$ is diagonalizable on $E(k_0,k_1,k_2,k_3)$ with exactly one eigenvalue $k_0$. 

\item If $d\geq 3$ then $t_0$ is diagonalizable on $E(k_0,k_1,k_2,k_3)$ with exactly two eigenvalues $k_0^{\pm 1}$. 
\end{enumerate}
\end{lem}
\begin{proof}
The statement (i) is immediate from Lemma \ref{lem1:t0_E}. By (\ref{k02=-d-1}) and since $q^{d+1}\not=1$ the values $k_0$ and $k_0^{-1}$ are distinct. Applying the rank-nullity theorem to Lemma \ref{lem1:t0_E} the statement (ii) follows.
\end{proof}

\begin{lem}\label{lem3:t0_E}
$E(k_0,k_1,k_2,k_3)(k_0)$ is of dimension $\frac{d+3}{2}$ with the $\F$-basis 
\begin{gather*}
v_0,
\quad 
v_d,
\quad 
v_i+(q^i-1) v_{i-1}
\quad \hbox{for $i=2,4,\ldots,d-1$}.
\end{gather*}
\end{lem}
\begin{proof}
Immediate from Lemma \ref{lem1:t0_E}.
\end{proof}

In light of Proposition \ref{prop:t0eigenspace=AWmodule}, $E(k_0,k_1,k_2,k_3)(k_0)$ is a $\triangle_q$-submodule of $E(k_0,k_1,k_2,k_3)$. We are now going to study the $\triangle_q$-module $E(k_0,k_1,k_2,k_3)(k_0)$ and the quotient $\triangle_q$-module of $E(k_0,k_1,k_2,k_3)$ modulo $E(k_0,k_1,k_2,k_3)(k_0)$.

\begin{lem}\label{lem:AB_E(k0)}
Let 
$$
\mu_i=(-1)^{\frac{i}{2}}k_0^{-\frac{i}{2}}  
k_1^{-\frac{i}{2}} 
q^{-\frac{i(i+2)}{4}}
\qquad 
\hbox{for $i=2,4,\ldots,d+1$}.
$$  
Then the matrices representing $A$ and $B$ with respect to the $\F$-basis 
\begin{gather}\label{basis:E(k0)}
v_0,
\quad 
\mu_i
(v_i+(q^i-1) v_{i-1})
\quad \hbox{for $i=2,4,\ldots,d-1$},
\quad 
\mu_{d+1}
(q^{d+1}-1) v_d
\end{gather}
for the $\triangle_q$-module $E(k_0,k_1,k_2,k_3)(k_0)$ are 
$$
\begin{pmatrix}
\theta_0 & & &  &{\bf 0}
\\ 
1 &\theta_1 
\\
&1 &\theta_2
 \\
& &\ddots &\ddots
 \\
{\bf 0} & & &1 &\theta_\frac{d+1}{2}
\end{pmatrix},
\qquad 
\begin{pmatrix}
\theta_0^* &\varphi_1 &  & &{\bf 0}
\\ 
 &\theta_1^* &\varphi_2
\\
 &  &\theta_2^* &\ddots
 \\
 & & &\ddots &\varphi_{\frac{d+1}{2}}
 \\
{\bf 0}  & & & &\theta_\frac{d+1}{2}^*
\end{pmatrix},
$$
respectively, where 
\begin{align*}
\theta_i &=
k_0 k_1 q^{2i}
+k_0^{-1} k_1^{-1} q^{-2i}
\qquad 
\hbox{for $i=0,1,\ldots, \textstyle\frac{d+1}{2}$},
\\
\theta_i^* &=
k_0 k_3 q^{2i}
+k_0^{-1} k_3^{-1} q^{-2i}
\qquad 
\hbox{for $i=0,1,\ldots, \textstyle\frac{d+1}{2}$},
\\
\varphi_i &=
k_1^{-1} k_3^{-1} q^{\frac{d+3}{2}}
(q^i-q^{-i})(q^{i-\frac{d+3}{2}}-q^{\frac{d+3}{2}-i})
\\
&\qquad \times \;
(q^{-i}-k_0 k_1 k_2 k_3 q^{i-1})
(q^{-i}-k_0 k_1 k_2^{-1} k_3 q^{i-1})
\qquad 
\hbox{for $i=1,2,\ldots, \textstyle\frac{d+1}{2}$}.
\end{align*}
The elements $\alpha,\beta,\gamma$ act on the $\triangle_q$-module $E(k_0,k_1,k_2,k_3)(k_0)$ as scalar multiplication by 
\begin{gather}
(k_3+k_3^{-1})(k_2+k_2^{-1})+(k_1+k_1^{-1})(q k_0^{-1}+q^{-1} k_0),
\label{alpha:E(k0)}
\\
(k_2+k_2^{-1})(k_1+k_1^{-1})+(k_3+k_3^{-1})(q k_0^{-1}+q^{-1} k_0),
\label{beta:E(k0)}
\\
(k_1+k_1^{-1})(k_3+k_3^{-1})+(k_2+k_2^{-1})(q k_0^{-1}+q^{-1} k_0),
\label{gam:E(k0)}
\end{gather}
respectively.
\end{lem}
\begin{proof}
By Lemma \ref{lem3:t0_E} the vectors (\ref{basis:E(k0)}) are an $\F$-basis for $E(k_0,k_1,k_2,k_3)(k_0)$. Applying Lemma \ref{lem:AB_E} a direct calculation yields the matrices representing 
$A$ and $B$ with respect to (\ref{basis:E(k0)}). By Theorem \ref{thm:hom} the elements $\alpha,\beta,\gamma$ act on $E(k_0,k_1,k_2,k_3)(k_0)$ as scalar multiplication by (\ref{alpha:E(k0)})--(\ref{gam:E(k0)}), respectively. 
\end{proof}

\begin{prop}\label{prop:E(k0)}
The $\triangle_q$-module $E(k_0,k_1,k_2,k_3)(k_0)$ is isomorphic to 
$$
V_{\frac{d+1}{2}}\left(
k_0 k_1 q^{\frac{d+1}{2}},
k_0 k_3 q^{\frac{d+1}{2}},
k_0 k_2 q^{\frac{d+1}{2}}
\right).
$$
Moreover  if the $\H_q$-module $E(k_0,k_1,k_2,k_3)$ is irreducible then the $\triangle_q$-module $E(k_0,k_1,k_2,k_3)(k_0)$ is irreducible.
\end{prop}
\begin{proof}
Set $(a,b,c)=(k_0 k_1 q^{\frac{d+1}{2}},
k_0 k_3 q^{\frac{d+1}{2}},
k_0 k_2 q^{\frac{d+1}{2}})$ and $d'=\frac{d+1}{2}$. Under the assumption (\ref{k02=-d-1}) the scalar (\ref{beta:E(k0)}) is equal to
$$
(c+c^{-1})(a+a^{-1})+(b+b^{-1})(q^{d'+1}+q^{-d'-1}).
$$
Comparing Proposition \ref{prop:UAWd} with Lemma \ref{lem:AB_E(k0)} we see that the $\triangle_q$-module $E(k_0,k_1,k_2,k_3)(k_0)$ is isomorphic to $V_{d'}(a,b,c)$. 

Suppose that the $\H_q$-module $E(k_0,k_1,k_2,k_3)$ is irreducible. Using Theorem \ref{thm:irr_E} yields that 
$$
abc,a^{-1}bc,ab^{-1}c,abc^{-1}\not\in\{q^{2i-d'-1}\,|\,i=1,2,\ldots,d'\}.
$$
Hence the $\triangle_q$-module $V_{d'}(a,b,c)$ is irreducible by Theorem \ref{thm:irr_UAW}. The proposition follows.
\end{proof}

\begin{lem}\label{lem:AB_E/E(k0)}
Suppose that $d\geq 3$. Let  
$$
\mu_i=(-1)^{\frac{i-1}{2}} k_0^{-\frac{i-1}{2}}  k_1^{-\frac{i-1}{2}}  q^{-\frac{(i-1)(i+1)}{2}}
\qquad  
\hbox{for $i=1,3,\ldots,d-2$}. 
$$
Then the matrices representing $A$ and $B$ with respect to the $\F$-basis 
\begin{gather}\label{basis:E/E(k0)}
\mu_i v_i+E(k_0,k_1,k_2,k_3)(k_0)
\qquad \hbox{for $i=1,3,\ldots,d-2$}
\end{gather}
for the $\triangle_q$-module $E(k_0,k_1,k_2,k_3)/E(k_0,k_1,k_2,k_3)(k_0)$ are 
$$
\begin{pmatrix}
\theta_0 & & &  &{\bf 0}
\\ 
1 &\theta_1 
\\
&1 &\theta_2
 \\
& &\ddots &\ddots
 \\
{\bf 0} & & &1 &\theta_\frac{d-3}{2}
\end{pmatrix},
\qquad 
\begin{pmatrix}
\theta_0^* &\varphi_1 &  & &{\bf 0}
\\ 
 &\theta_1^* &\varphi_2
\\
 &  &\theta_2^* &\ddots
 \\
 & & &\ddots &\varphi_{\frac{d-3}{2}}
 \\
{\bf 0}  & & & &\theta_\frac{d-3}{2}^*
\end{pmatrix},
$$
respectively, where  
\begin{align*}
\theta_i &= k_0 k_1 q^{2i+2}+ k_0^{-1} k_1^{-1} q^{-2i-2}
\qquad 
\hbox{for $i=0,1,\ldots,\textstyle\frac{d-3}{2}$},
\\
\theta_i^* &= k_0 k_3 q^{2i+2}+ k_0^{-1} k_3^{-1} q^{-2i-2}
\qquad 
\hbox{for $i=0,1,\ldots,\textstyle\frac{d-3}{2}$},
\\
\varphi_i &=
k_1^{-1} k_3^{-1} q^{\frac{d-1}{2}}
(q^i-q^{-i})
(q^{i-\frac{d-1}{2}}-q^{\frac{d-1}{2}-i})
\\
&\qquad \times \;
(q^{-i}-k_0 k_1 k_2 k_3 q^{i+1})
(q^{-i}-k_0 k_1 k_2^{-1} k_3 q^{i+1})
\qquad 
\hbox{for $i=1,2,\ldots,\textstyle\frac{d-3}{2}$}.
\end{align*}
The elements $\alpha,\beta,\gamma$ act on the $\triangle_q$-module $E(k_0,k_1,k_2,k_3)/E(k_0,k_1,k_2,k_3)(k_0)$ as scalar multiplication by 
\begin{gather}
(k_3+k_3^{-1})(k_2+k_2^{-1})+(k_1+k_1^{-1})(q k_0+q^{-1}k_0^{-1}),
\label{alpha:E/E(k0)}
\\
(k_2+k_2^{-1})(k_1+k_1^{-1})+(k_3+k_3^{-1})(q k_0+q^{-1}k_0^{-1}),
\label{beta:E/E(k0)}
\\
(k_1+k_1^{-1})(k_3+k_3^{-1})+(k_2+k_2^{-1})(q k_0+q^{-1}k_0^{-1}),
\label{gam:E/E(k0)}
\end{gather}
respectively.
\end{lem}
\begin{proof}
By Lemma \ref{lem3:t0_E} the cosets (\ref{basis:E/E(k0)}) are an $\F$-basis for $E(k_0,k_1,k_2,k_3)/E(k_0,k_1,k_2,k_3)(k_0)$. 
Applying Lemmas \ref{lem:AB_E} and \ref{lem3:t0_E} a direct calculation yields the matrices representing 
$A$ and $B$ with respect to (\ref{basis:E/E(k0)}). By Lemma \ref{lem1:t0_E} we have 
\begin{gather*}
(t_0-k_0^{-1})v_i \in E(k_0,k_1,k_2,k_3)(k_0)
\qquad 
\hbox{for $i=1,3,\ldots,d-2$}.
\end{gather*}
Combined with Theorem \ref{thm:hom} and Proposition \ref{prop:E}(ii) this yields that $\alpha,\beta,\gamma$ act on the quotient of $E(k_0,k_1,k_2,k_3)$ modulo $E(k_0,k_1,k_2,k_3)(k_0)$ as scalar multiplication by (\ref{alpha:E/E(k0)})--(\ref{gam:E/E(k0)}), respectively. 
\end{proof}

\begin{prop}\label{prop:E/E(k0)}
Suppose that $d\geq 3$. Then the quotient $\triangle_q$-module of $E(k_0,k_1,k_2,k_3)$ modulo $E(k_0,k_1,k_2,k_3)(k_0)$ is  isomorphic to 
$$
V_{\frac{d-3}{2}}\left(
k_0 k_1 q^{\frac{d+1}{2}},
k_0 k_3 q^{\frac{d+1}{2}},
k_0 k_2 q^{\frac{d+1}{2}}
\right).
$$
Moreover the $\triangle_q$-module $E(k_0,k_1,k_2,k_3)/E(k_0,k_1,k_2,k_3)(k_0)$ is irreducible if the $\H_q$-module $E(k_0,k_1,k_2,k_3)$ is irreducible.
\end{prop}
\begin{proof}
Set $(a,b,c)=(k_0 k_1 q^{\frac{d+1}{2}},
k_0 k_3 q^{\frac{d+1}{2}},
k_0 k_2 q^{\frac{d+1}{2}})$ and $d'=\frac{d-3}{2}$. Under the assumption (\ref{k02=-d-1}) the scalar (\ref{beta:E/E(k0)}) is equal to 
$$
(c+c^{-1})(a+a^{-1})+(b+b^{-1})(q^{d'+1}+q^{-d'-1}).
$$
Comparing Proposition \ref{prop:UAWd} with Lemma \ref{lem:AB_E/E(k0)} yields that the quotient $\triangle_q$-module of $E(k_0,k_1,k_2,k_3)$ modulo $E(k_0,k_1,k_2,k_3)(k_0)$ is isomorphic to $V_{d'}(a,b,c)$. 

Suppose that the $\H_q$-module $E(k_0,k_1,k_2,k_3)$ is irreducible. Using Theorem \ref{thm:irr_E} yields that 
$$
abc,a^{-1}bc,ab^{-1}c,abc^{-1}\not\in\{q^{2i-d'-1}\,|\,i=0,1,\ldots,d'+1\}.
$$
Hence the $\triangle_q$-module $V_{d'}(a,b,c)$ is irreducible by Theorem \ref{thm:irr_UAW}. The proposition follows.
\end{proof}

\begin{thm}\label{thm:E0}
Assume that the $\H_q$-module $E(k_0,k_1,k_2,k_3)$ is irreducible. Then the following hold:
\begin{enumerate}
\item If $d=1$ then the $\triangle_q$-module $E(k_0,k_1,k_2,k_3)$ is irreducible.

\item If $d\geq 3$ then
\begin{table}[H]
\begin{tikzpicture}[node distance=1.2cm]
 \node (0)                  {$\{0\}$};
 \node (1)  [above of=0]   {};
 \node (2)  [right of=1]   {};
 \node (3)  [left of=1]   {};
 \node (E1)  [right of=2]  {$E(k_0,k_1,k_2,k_3)(k_0)$};
 \node (E2)  [left of=3]   {$E(k_0,k_1,k_2,k_3)(k_0^{-1})$};
 \node (V) [above of=1]  {$E(k_0,k_1,k_2,k_3)$};
 \draw (0)   -- (E1);
 \draw (0)   -- (E2);
 \draw (E1)   -- (V);
 \draw (E2)  -- (V);
\end{tikzpicture}
\end{table}
\noindent is the lattice of $\triangle_q$-submodules of $E(k_0,k_1,k_2,k_3)$.

\end{enumerate}
\end{thm}
\begin{proof}
(i): Suppose that $d=1$. Then $E(k_0,k_1,k_2,k_3)=E(k_0,k_1,k_2,k_3)(k_0)$ by Lemma \ref{lem2:t0_E}(i).  It 
follows from Proposition \ref{prop:E(k0)} that the $\triangle_q$-module $E(k_0,k_1,k_2,k_3)$ is irreducible. The statement (i) follows.

(ii): Suppose that $d\geq 3$. By Propositions \ref{prop:E(k0)} and \ref{prop:E/E(k0)} the sequence
\begin{gather}\label{E:series1}
\{0\} \subset E(k_0,k_1,k_2,k_3)(k_0) \subset E(k_0,k_1,k_2,k_3)
\end{gather}
is a composition series for the $\triangle_q$-module $E(k_0,k_1,k_2,k_3)$. By Proposition \ref{prop:t0eigenspace=AWmodule} and Lemma \ref{lem2:t0_E}(ii), $E(k_0,k_1,k_2,k_3)(k_0^{-1})$ is a nonzero $\triangle_q$-submodule of $E(k_0,k_1,k_2,k_3)$. Hence 
\begin{gather}\label{E:series2}
\{0\} \subset E(k_0,k_1,k_2,k_3)(k_0^{-1}) \subset E(k_0,k_1,k_2,k_3)
\end{gather}
is a composition series for the $\triangle_q$-module $E(k_0,k_1,k_2,k_3)$ by Jordan--H\"{o}lder theorem. By Proposition \ref{prop:irr_AWmodule_in_t0eigenspace} there is no other irreducible $\triangle_q$-submodule of $E(k_0,k_1,k_2,k_3)$. Therefore (\ref{E:series1}) and (\ref{E:series2}) are the unique two composition series for the $\triangle_q$-module $E(k_0,k_1,k_2,k_3)$. The statement (ii) follows.
\end{proof}

\subsection{The lattice of $\triangle_q$-submodules of $E(k_0,k_1,k_2,k_3)^{1 \bmod{4}}$}\label{s:lattice_E1}

Recall from Table \ref{Z/4Z-action} that $1\bmod{4}$ sends $t_0,t_1,t_2,t_3$ to $t_1,t_2,t_3,t_0$, respectively. 
In this subsection, we study the $\triangle_q$-submodules of the $\H_q$-module $E(k_0,k_1,k_2,k_3)^{1 \bmod{4}}$.

\begin{lem}\label{lem:AB_E1}
The actions of $A$ and $B$ on the $\H_q$-module $E(k_0,k_1,k_2,k_3)^{1 \bmod{4}}$ are as follows:
\begin{align*}
A v_i &=
\left\{
\begin{array}{ll}
\displaystyle
\theta_i^* v_i 
+
k_0^{-1} k_3^{-1} q^{-i} \varrho_i (q-q^{-1}) v_{i-1}
-
k_0^{-1} k_3^{-1} q^{1-i}\varrho_i \varrho_{i-1} v_{i-2}
\qquad 
&\hbox{for $i=2,4,\ldots,d-1$},
\\
\displaystyle
\theta_i^* v_i 
-
k_0^{-1} k_3^{-1} q^{-i}\varrho_i \varrho_{i-1} v_{i-2}
\qquad 
&\hbox{for $i=3,5,\ldots,d$},
\end{array}
\right.
\\
A v_0&=\theta_0^* v_0,
\qquad 
A v_1 = \theta_1^* v_1,
\\
B v_i &=
\left\{
\begin{array}{ll}
\displaystyle
\theta_i v_i 
-
k_0^{-1} k_1^{-1} q^{-i-1}(q-q^{-1}) v_{i+1}
-
k_0^{-1} k_1^{-1} q^{-i-2} v_{i+2}
\qquad 
&\hbox{for $i=0,2,\ldots,d-3$},
\\
\displaystyle
\theta_i v_i -k_0^{-1} k_1^{-1} q^{-i-1} v_{i+2}
\qquad 
&\hbox{for $i=1,3,\ldots,d-2$},
\end{array}
\right.
\\
B v_{d-1}&=\theta_{d-1} v_{d-1} -k_0^{-1} k_1^{-1} q^{-d} (q-q^{-1}) v_d,
\qquad 
B v_d = \theta_d v_d,
\end{align*}
where 
\begin{align*}
\theta_i
&=
k_0 k_1 q^{2\lceil \frac{i}{2}\rceil}
+k_0^{-1} k_1^{-1}  q^{-2\lceil \frac{i}{2}\rceil}
\qquad 
\hbox{for $i=0,1,\ldots,d$},
\\
\theta_i^*
&=
k_0 k_3 q^{2\lfloor \frac{i}{2}\rfloor +1}
+
k_0^{-1} k_3^{-1} q^{-2\lfloor \frac{i}{2}\rfloor -1}
\qquad 
\hbox{for $i=0,1,\ldots,d$}.
\end{align*}
\end{lem}
\begin{proof}
By Lemma \ref{lem:X+Xinv} the actions of $A$ and $B$ on $E(k_0,k_1,k_2,k_3)^{1 \bmod{4}}$ are identical to the actions of 
$q X+q^{-1} X^{-1}$ and $Y+Y^{-1}$
on $E(k_0,k_1,k_2,k_3)$, respectively. Using Lemma \ref{lem:XYinE} it is routine to evaluate the actions of $q X+q^{-1} X^{-1}$ and $Y+Y^{-1}$ on $E(k_0,k_1,k_2,k_3)$.
\end{proof}

\begin{lem}\label{lem1:t0_E1}
The matrix representing $t_0$ with respect to the $\F$-basis 
\begin{gather*}
v_1,
\quad 
v_{i+1}-k_1^2\varrho_i v_{i-1}
\quad 
\hbox{for $i=2,4,\ldots,d-1$},
\quad 
v_i 
\quad 
\hbox{for $i=0,2,\ldots,d-1$}
\end{gather*}
for $E(k_0,k_1,k_2,k_3)^{1 \bmod{4}}$ is 
\begin{gather*}
\begin{pmatrix}
k_1^{-1} I_{\frac{d+1}{2}} & \rvline & k_1^{-1} I_{\frac{d+1}{2}}
\\
\hline
{\bf 0} & \rvline &k_1 I_{\frac{d+1}{2}} 
\end{pmatrix}.
\end{gather*}
\end{lem}
\begin{proof}
Note that the action of $t_0$ on $E(k_0,k_1,k_2,k_3)^{1 \bmod{4}}$ is identical to the action of $t_1$ on $E(k_0,k_1,k_2,k_3)$. It is routine to verify the lemma using Proposition \ref{prop:E}(i).
\end{proof}

\begin{lem}\label{lem2:t0_E1}
\begin{enumerate}
\item If $k_1^2=1$ then $t_0$ is not diagonalizable on $E(k_0,k_1,k_2,k_3)^{1 \bmod{4}}$ with exactly one eigenvalue $k_1$.

\item If $k_1^2\not=1$ then $t_0$ is diagonalizable on $E(k_0,k_1,k_2,k_3)^{1 \bmod{4}}$ with exactly two eigenvalues $k_1^{\pm 1}$. 
\end{enumerate}
\end{lem}
\begin{proof}
Applying the rank-nullity theorem to Lemma \ref{lem1:t0_E1} the lemma follows.
\end{proof}

\begin{lem}\label{lem3:t0_E1}
$E(k_0,k_1,k_2,k_3)^{1 \bmod{4}}(k_1^{-1})$ is of dimension $\frac{d+1}{2}$ with the 
$\F$-basis 
\begin{gather}\label{basis:E1(k1inv)}
v_i
\qquad 
\hbox{for $i=1,3,\ldots,d$}.
\end{gather}
\end{lem}
\begin{proof}
By Lemma \ref{lem1:t0_E1} the vectors $v_1$ and
$v_{i+1}-k_1^2\varrho_i v_{i-1}$ for $i=2,4,\ldots,d-1$ form an $\F$-basis for $E(k_0,k_1,k_2,k_3)^{1 \bmod{4}}(k_1^{-1})$ as well as (\ref{basis:E1(k1inv)}).
\end{proof}

\begin{lem}\label{lem:AB_E1(k1inv)}
Let 
$$
\mu_i=(-1)^{\frac{i-1}{2}} k_0^{-\frac{i-1}{2}} k_1^{-{\frac{i-1}{2}}} q^{-\frac{(i-1)(i+1)}{4}}
\qquad 
\hbox{for $i=1,3,\ldots,d$}.
$$
Then the matrices representing $A$ and $B$ with respect to the $\F$-basis 
\begin{gather}\label{basis_E1(k1inv)}
\mu_i v_i
\qquad 
\hbox{for $i=1,3,\ldots,d$}
\end{gather}
for the $\triangle_q$-module $E(k_0,k_1,k_2,k_3)^{1 \bmod{4}}(k_1^{-1})$ are 
$$
\begin{pmatrix}
\theta_0^* &\varphi_1 &  & &{\bf 0}
\\ 
 &\theta_1^* &\varphi_2
\\
 &  &\theta_2^* &\ddots
 \\
 & & &\ddots &\varphi_{\frac{d-1}{2}}
 \\
{\bf 0}  & & & &\theta_\frac{d-1}{2}^*
\end{pmatrix},
\qquad 
\begin{pmatrix}
\theta_0 & & &  &{\bf 0}
\\ 
1 &\theta_1 
\\
&1 &\theta_2
 \\
& &\ddots &\ddots
 \\
{\bf 0} & & &1 &\theta_\frac{d-1}{2}
\end{pmatrix},
$$
respectively, where 
\begin{align*}
\theta_i
&=
k_0 k_1 q^{2i+2}+k_0^{-1} k_1^{-1} q^{-2i-2}
\qquad 
\hbox{for $i=0,1,\ldots,\textstyle\frac{d-1}{2}$},
\\
\theta_i^*
&=
k_0 k_3 q^{2i+1}+ k_0^{-1} k_3^{-1} q^{-2i-1}
\qquad 
\hbox{for $i=0,1,\ldots,\textstyle\frac{d-1}{2}$},
\\
\varphi_i &=
k_1^{-1} k_3^{-1} q^{\frac{d-1}{2}} 
(q^i-q^{-i})
(q^{i-\frac{d+1}{2}}-q^{\frac{d+1}{2}-i})
\\
&\qquad 
\times \;
(q^{-i}-k_0 k_1 k_2 k_3 q^{i+1})
(q^{-i}-k_0 k_1 k_2^{-1} k_3 q^{i+1})
\qquad 
\hbox{for $i=1,2,\ldots, \textstyle\frac{d-1}{2}$}.
\end{align*}
The elements $\alpha,\beta,\gamma$ act on the $\triangle_q$-module $E(k_0,k_1,k_2,k_3)^{1 \bmod{4}}(k_1^{-1})$ as scalar multiplication by 
\begin{gather}
(k_0 +k_0^{-1})
(k_3+k_3^{-1})
+
(k_2+k_2^{-1})
(q k_1+q^{-1} k_1^{-1}),
\label{alpha:E1(k1inv)}
\\
(k_3+k_3^{-1})
(k_2+k_2^{-1})
+
(k_0 +k_0^{-1})
(q k_1+q^{-1} k_1^{-1}),
\label{beta:E1(k1inv)}
\\
(k_2+k_2^{-1})
(k_0 +k_0^{-1})
+
(k_3+k_3^{-1})
(q k_1+q^{-1} k_1^{-1}),
\label{gamma:E1(k1inv)}
\end{gather}
respectively.
\end{lem}
\begin{proof}
By Lemma \ref{lem3:t0_E1} the vectors (\ref{basis_E1(k1inv)}) are a basis for $E(k_0,k_1,k_2,k_3)^{1 \bmod{4}}(k_1^{-1})$. Using Lemma \ref{lem:AB_E1}  a routine calculation yields the matrices representing $A$ and $B$ with respect to (\ref{basis_E1(k1inv)}). The actions of $c_0,c_1,c_2,c_3$ on $E(k_0,k_1,k_2,k_3)^{1 \bmod{4}}$ are identical to the actions of $c_1,c_2,c_3,c_0$ on $E(k_0,k_1,k_2,k_3)$. 
By Theorem \ref{thm:hom} the elements $\alpha,\beta,\gamma$ act on $E(k_0,k_1,k_2,k_3)^{1 \bmod{4}}(k_1^{-1})$ as scalar multiplication by (\ref{alpha:E1(k1inv)})--(\ref{gamma:E1(k1inv)}), respectively.
\end{proof}

Recall the $\Z/2\Z$-action on $\triangle_q$ from Table \ref{Z/2Z-action}.

\begin{prop}\label{prop:E1(k1inv)}
The $\triangle_q$-module $E(k_0,k_1,k_2,k_3)^{1 \bmod{4}}(k_1^{-1})$ is isomorphic to $$
V_{\frac{d-1}{2}}\left(
k_0 k_1 q^{\frac{d+3}{2}},
k_0 k_3 q^{\frac{d+1}{2}},
k_0 k_2 q^{\frac{d+1}{2}}
\right)^{1 \bmod{2}}.
$$
Moreover the $\triangle_q$-module $E(k_0,k_1,k_2,k_3)^{1 \bmod{4}}(k_1^{-1})$ is irreducible and it is isomorphic to $V_{\frac{d-1}{2}}(
k_0 k_3 q^{\frac{d+1}{2}},
k_0 k_1 q^{\frac{d+3}{2}},
k_0 k_2 q^{\frac{d+1}{2}})$ provided that the $\H_q$-module $E(k_0,k_1,k_2,k_3)$ is irreducible.
\end{prop}
\begin{proof}
Set $(a,b,c)=(k_0 k_1 q^{\frac{d+3}{2}},
k_0 k_3 q^{\frac{d+1}{2}},
k_0 k_2 q^{\frac{d+1}{2}})$
and $d'=\frac{d-1}{2}$. Under the assumption (\ref{k02=-d-1}) the scalar (\ref{beta:E1(k1inv)}) is equal to 
$$
(b+b^{-1})(c+c^{-1})+(a+a^{-1})(q^{d'+1}+q^{-d'-1}).
$$
By Proposition \ref{prop:UAWd} and Lemma \ref{lem:AB_E1(k1inv)} the $\triangle_q$-module $E(k_0,k_1,k_2,k_3)^{1 \bmod{4}}(k_1^{-1})$ is isomorphic to $V_{d'}(a,b,c)^{1 \bmod{2}}$.

Suppose that the $\H_q$-module $E(k_0,k_1,k_2,k_3)$ is irreducible. Using Theorem \ref{thm:irr_E} yields that 
\begin{align*}
&a^{-1}bc\not\in \{q^{2i-d'-1}\,|\,i=0,1,\ldots,d'\};
\\
&abc,ab^{-1}c,abc^{-1}\not\in\{q^{2i-d'-1}\,|\,i=1,2,\ldots,d'+1\}.
\end{align*}
Hence the $\triangle_q$-module $V_{d'}(a,b,c)$ is irreducible by Theorem \ref{thm:irr_UAW}. 
By Lemma \ref{lem:Z/2Z} the $\triangle_q$-module $V_{d'}(a,b,c)^{1 \bmod{2}}$ is isomorphic to $V_{d'}(b,a,c)$.
The proposition follows.
\end{proof}

\begin{lem}\label{lem:AB_E1/E1(k1inv)}
Let 
$$
\mu_i=(-1)^{\frac{i}{2}} k_0^{-\frac{i}{2}} k_1^{-\frac{i}{2}} q^{-\frac{i(i+2)}{4}}
\qquad  
\hbox{for $i=0,2,\ldots,d-1$}.
$$
Then the matrices representing $A$ and $B$ with respect to the $\F$-basis 
\begin{gather}\label{basis_E1/E1(k1inv)}
\mu_i v_i+E(k_0,k_1,k_2,k_3)^{1 \bmod{4}}(k_1^{-1})
\qquad 
\hbox{for $i=0,2,\ldots,d-1$}
\end{gather}
for the $\triangle_q$-module $E(k_0,k_1,k_2,k_3)^{1 \bmod{4}}/E(k_0,k_1,k_2,k_3)^{1 \bmod{4}}(k_1^{-1})$ are 
$$
\begin{pmatrix}
\theta_0^* &\varphi_1 &  & &{\bf 0}
\\ 
 &\theta_1^* &\varphi_2
\\
 &  &\theta_2^* &\ddots
 \\
 & & &\ddots &\varphi_{\frac{d-1}{2}}
 \\
{\bf 0}  & & & &\theta_\frac{d-1}{2}^*
\end{pmatrix},
\qquad 
\begin{pmatrix}
\theta_0 & & &  &{\bf 0}
\\ 
1 &\theta_1 
\\
&1 &\theta_2 
 \\
& &\ddots &\ddots
 \\
{\bf 0} & & &1 &\theta_\frac{d-1}{2}
\end{pmatrix},
$$
respectively, where 
\begin{align*}
\theta_i
&=k_0 k_1 q^{2i}+k_0^{-1} k_1^{-1} q^{-2i}
\qquad 
\hbox{for $i=0,1,\ldots,\textstyle\frac{d-1}{2}$},
\\
\theta_i^*
&=k_0 k_3 q^{2i+1}+k_0^{-1} k_3^{-1} q^{-2i-1}
\qquad 
\hbox{for $i=0,1,\ldots,\textstyle\frac{d-1}{2}$},
\\
\varphi_i &=
k_1^{-1} k_3^{-1} q^{\frac{d+3}{2}}
(q^i-q^{-i})(q^{i-\frac{d+1}{2}}-q^{\frac{d+1}{2}-i})
\\
&\qquad \times \;
(q^{-i}-k_0 k_1 k_2 k_3 q^{i-1})
(q^{-i}-k_0 k_1 k_2^{-1} k_3 q^{i-1})
\qquad 
\hbox{for $i=1,2,\ldots,\textstyle\frac{d-1}{2}$}.
\end{align*}
The elements $\alpha,\beta,\gamma$ act on the $\triangle_q$-module $E(k_0,k_1,k_2,k_3)^{1 \bmod{4}}/E(k_0,k_1,k_2,k_3)^{1 \bmod{4}}(k_1^{-1})$ as scalar multiplication by
\begin{gather}
(k_0 +k_0^{-1})
(k_3+k_3^{-1})
+
(k_2+k_2^{-1})
(q k_1^{-1}+q^{-1} k_1),
\label{alpha:E1/E1(k1inv)}
\\
(k_3+k_3^{-1})
(k_2+k_2^{-1})
+
(k_0 +k_0^{-1})
(q k_1^{-1}+q^{-1} k_1),
\label{beta:E1/E1(k1inv)}
\\
(k_2+k_2^{-1})
(k_0 +k_0^{-1})
+
(k_3+k_3^{-1})
(q k_1^{-1}+q^{-1} k_1),
\label{gamma:E1/E1(k1inv)}
\end{gather}
respectively.
\end{lem}
\begin{proof}
By Lemma \ref{lem3:t0_E1} the cosets (\ref{basis_E1/E1(k1inv)}) are a basis for the quotient of $E(k_0,k_1,k_2,k_3)^{1 \bmod{4}}$ modulo $E(k_0,k_1,k_2,k_3)^{1 \bmod{4}}(k_1^{-1})$. By Lemmas \ref{lem:AB_E1} and \ref{lem3:t0_E1} we obtain the matrices representing $A$ and $B$ with respect to (\ref{basis_E1/E1(k1inv)}). By Lemma \ref{lem1:t0_E1}
we have 
\begin{gather*}
(t_0-k_1)v_i\in E(k_0,k_1,k_2,k_3)^{1 \bmod{4}}(k_1^{-1})
\qquad 
\hbox{for $i=0,2,\ldots,d-1$}.
\end{gather*}
Combined with Theorem \ref{thm:hom} and Proposition \ref{prop:E}(ii) this yields that $\alpha,\beta,\gamma$ act on the quotient $\triangle_q$-module of $E(k_0,k_1,k_2,k_3)^{1 \bmod{4}}$ modulo $E(k_0,k_1,k_2,k_3)^{1 \bmod{4}}(k_1^{-1})$ as scalar multiplication by (\ref{alpha:E1/E1(k1inv)})--(\ref{gamma:E1/E1(k1inv)}), respectively.
\end{proof}

\begin{prop}\label{prop:E1/E1(k1inv)}
The $\triangle_q$-module $E(k_0,k_1,k_2,k_3)^{1 \bmod{4}}/E(k_0,k_1,k_2,k_3)^{1 \bmod{4}}(k_1^{-1})$ is isomorphic to 
$$
V_{\frac{d-1}{2}}\left(
k_0 k_1 q^{\frac{d-1}{2}},
k_0 k_3 q^{\frac{d+1}{2}},
k_0 k_2 q^{\frac{d+1}{2}}
\right)^{1 \bmod{2}}.
$$
Moreover the $\triangle_q$-module 
$
E(k_0,k_1,k_2,k_3)^{1 \bmod{4}}/E(k_0,k_1,k_2,k_3)^{1 \bmod{4}}(k_1^{-1})
$ 
is irreducible and it is isomorphic to 
$V_{\frac{d-1}{2}}(
k_0 k_3 q^{\frac{d+1}{2}},
k_0 k_1 q^{\frac{d-1}{2}},
k_0 k_2 q^{\frac{d+1}{2}})$ if the $\H_q$-module 
$E(k_0,k_1,k_2,k_3)$ is irreducible.
\end{prop}
\begin{proof}
Set $(a,b,c)=(k_0 k_1 q^{\frac{d-1}{2}},
k_0 k_3 q^{\frac{d+1}{2}},
k_0 k_2 q^{\frac{d+1}{2}})$
and $d'=\frac{d-1}{2}$. Under the assumption (\ref{k02=-d-1}) the scalar (\ref{beta:E1/E1(k1inv)}) is equal to 
$$
(b+b^{-1})(c+c^{-1})+(a+a^{-1})(q^{d'+1}+q^{-d'-1}).
$$
By Proposition \ref{prop:UAWd} and Lemma \ref{lem:AB_E1/E1(k1inv)} the quotient $\triangle_q$-module $E(k_0,k_1,k_2,k_3)^{1 \bmod{4}}$ modulo $E(k_0,k_1,k_2,k_3)^{1 \bmod{4}}(k_1^{-1})$ is isomorphic to $V_{d'}(a,b,c)^{1 \bmod{2}}$.

Suppose that the $\H_q$-module $E(k_0,k_1,k_2,k_3)$ is irreducible. Using Theorem \ref{thm:irr_E} yields that 
\begin{align*}
&a^{-1}bc\not\in\{q^{2i-d'-1}\,|\,i=1,2,\ldots,d'+1\};
\\
&abc,ab^{-1}c,abc^{-1}\not\in\{q^{2i-d'-1}\,|\,i=0,1,\ldots,d'\}.
\end{align*}
Hence the $\triangle_q$-module $V_{d'}(a,b,c)$ is irreducible by Theorem \ref{thm:irr_UAW}. 
By Lemma \ref{lem:Z/2Z} the $\triangle_q$-module $V_{d'}(a,b,c)^{1 \bmod{2}}$ is isomorphic to $V_{d'}(b,a,c)$. 
The proposition follows.
\end{proof}

\begin{thm}\label{thm:E1}
Assume that the $\H_q$-module $E(k_0,k_1,k_2,k_3)^{1 \bmod{4}}$ is irreducible. Then the following hold:
\begin{enumerate}
\item If $k_1^2=1$ then 
\begin{table}[H]
\begin{tikzpicture}[node distance=1.2cm]
 \node (0)                  {$\{0\}$};
 \node (E)  [above of=0]   {$E(k_0,k_1,k_2,k_3)^{1 \bmod{4}}(k_1)$};
 \node (V)  [above of=E]   {$E(k_0,k_1,k_2,k_3)^{1 \bmod{4}}$};
 \draw (0)   -- (E);
 \draw (E)  -- (V);
\end{tikzpicture}
\end{table}
\noindent is the lattice of $\triangle_q$-submodules of $E(k_0,k_1,k_2,k_3)^{1 \bmod{4}}$.

\item If $k_1^2\not=1$ then 
 \begin{table}[H]
\begin{tikzpicture}[node distance=1.5cm]
 \node (0)                  {$\{0\}$};
 \node (1)  [above of=0]   {};
 \node (2)  [right of=1]   {};
 \node (3)  [left of=1]   {};
 \node (E1)  [right of=2]  {$E(k_0,k_1,k_2,k_3)^{1 \bmod{4}}(k_1)$};
 \node (E2)  [left of=3]   {$E(k_0,k_1,k_2,k_3)^{1 \bmod{4}}(k_1^{-1})$};
 \node (V) [above of=1]  {$E(k_0,k_1,k_2,k_3)^{1 \bmod{4}}$};
 \draw (0)   -- (E1);
 \draw (0)   -- (E2);
 \draw (E1)   -- (V);
 \draw (E2)  -- (V);
\end{tikzpicture}
\end{table}
\noindent is the lattice of $\triangle_q$-submodules of $E(k_0,k_1,k_2,k_3)^{1 \bmod{4}}$.

\end{enumerate}
\end{thm}
\begin{proof}
(i): Suppose that $k_1^2=1$. By Propositions \ref{prop:E1(k1inv)} and \ref{prop:E1/E1(k1inv)} the sequence 
\begin{gather}\label{E1:series1}
\{0\} \subset E(k_0,k_1,k_2,k_3)^{1 \bmod{4}}(k_1) \subset E(k_0,k_1,k_2,k_3)^{1 \bmod{4}}
\end{gather}
is a composition series for the $\triangle_q$-module $E(k_0,k_1,k_2,k_3)^{1 \bmod{4}}$. It follows from Proposition \ref{prop:irr_AWmodule_in_t0eigenspace} and Lemma \ref{lem2:t0_E1}(i) that every irreducible $\triangle_q$-submodule of $E(k_0,k_1,k_2,k_3)^{1 \bmod{4}}$ is contained in $E(k_0,k_1,k_2,k_3)^{1 \bmod{4}}(k_1)$. Therefore (\ref{E1:series1}) is the unique composition series for the $\triangle_q$-module $E(k_0,k_1,k_2,k_3)^{1 \bmod{4}}$. The statement (i) follows.

(ii): Similar to the proof of Theorem \ref{thm:E0}(ii).
\end{proof}

\subsection{The lattice of $\triangle_q$-submodules of $E(k_0,k_1,k_2,k_3)^{2 \bmod{4}}$}\label{s:lattice_E2}

Recall from Table \ref{Z/4Z-action} that $2\bmod{4}$ sends $t_0,t_1,t_2,t_3$ to $t_2,t_3,t_0,t_1$, respectively. 
In this subsection, we study the $\triangle_q$-submodules of the $\H_q$-module $E(k_0,k_1,k_2,k_3)^{2 \bmod{4}}$.

\begin{lem}\label{lem:AB_E2}
The actions of $A$ and $B$ on the $\H_q$-module $E(k_0,k_1,k_2,k_3)^{2 \bmod{4}}$ are as follows:
\begin{align*}
A v_i &=
\left\{
\begin{array}{ll}
\displaystyle
\theta_i v_i 
-
k_0^{-1} k_1^{-1} q^{-i-1} v_{i+2}
\qquad 
&\hbox{for $i=0,2,\ldots,d-3$},
\\
\displaystyle
\theta_i v_i
-
k_0^{-1} k_1^{-1} q^{-i-1}(q-q^{-1}) v_{i+1}
-
k_0^{-1} k_1^{-1} q^{-i-2} v_{i+2}
\qquad 
&\hbox{for $i=1,3,\ldots,d-2$},
\end{array}
\right.
\\
A v_{d-1}&=\theta_{d-1} v_{d-1},
\qquad 
A v_d = \theta_d v_d,
\\
B v_i &=
\left\{
\begin{array}{ll}
\displaystyle
\theta_i^* v_i 
+
k_0^{-1} k_3^{-1} q^{-i} \varrho_i (q-q^{-1}) v_{i-1}
-
k_0^{-1} k_3^{-1} q^{1-i}\varrho_i \varrho_{i-1} v_{i-2}
\qquad 
&\hbox{for $i=2,4,\ldots,d-1$},
\\
\displaystyle
\theta_i^* v_i 
-
k_0^{-1} k_3^{-1} q^{-i}\varrho_i \varrho_{i-1} v_{i-2}
\qquad 
&\hbox{for $i=3,5,\ldots,d$},
\end{array}
\right.
\\
B v_0&=\theta_0^* v_0,
\qquad 
B v_1 = \theta_1^* v_1,
\end{align*}
where 
\begin{align*}
\theta_i
&=
k_0 k_1 q^{2\lfloor \frac{i}{2}\rfloor +1}
+k_0^{-1} k_1^{-1}  q^{-2\lfloor \frac{i}{2}\rfloor -1}
\qquad 
\hbox{for $i=0,1,\ldots,d$},
\\
\theta_i^*
&=
k_0 k_3 q^{2\lfloor \frac{i}{2}\rfloor +1}
+
k_0^{-1} k_3^{-1} q^{-2\lfloor \frac{i}{2}\rfloor -1}
\qquad 
\hbox{for $i=0,1,\ldots,d$}.
\end{align*}
\end{lem}
\begin{proof}
By Lemma \ref{lem:X+Xinv} the actions of $A$ and $B$ on $E(k_0,k_1,k_2,k_3)^{2 \bmod{4}}$ are identical to the actions of 
$q Y+q^{-1} Y^{-1}$ and 
$q X+q^{-1} X^{-1}$
on $E(k_0,k_1,k_2,k_3)$, respectively. Using Lemma \ref{lem:XYinE} a direct calculation yields the actions of $q Y+q^{-1} Y^{-1}$ and $q X+q^{-1} X^{-1}$ on $E(k_0,k_1,k_2,k_3)$.
\end{proof}

\begin{lem}\label{lem1:t0_E2}
\begin{enumerate}
\item $E(k_0,k_1,k_2,k_3)^{2 \bmod{4}}(k_2)$ is of dimension $\frac{d+1}{2}$ with the $\F$-basis 
\begin{gather}\label{basis:E2(k2)}
(1-k_0 k_1 k_2^{-1} k_3 q^i ) v_{i-1}-v_i
\quad 
\hbox{for $i=1,3,\ldots,d$}.
\end{gather}

\item $E(k_0,k_1,k_2,k_3)^{2 \bmod{4}}(k_2^{-1})$ is of dimension $\frac{d+1}{2}$ with the $\F$-basis 
\begin{gather}\label{basis:E2(k2inv)}
(1-k_0 k_1 k_2 k_3 q^i ) v_{i-1}-v_i
\quad 
\hbox{for $i=1,3,\ldots,d$}.
\end{gather}
\end{enumerate}
\end{lem}
\begin{proof}
Note that the action of $t_0$ on $E(k_0,k_1,k_2,k_3)^{2 \bmod{4}}$ is identical to the action of $t_2$ on $E(k_0,k_1,k_2,k_3)$. For $i=1,3,\ldots,d$ 
let $W_i$ denote the $\F$-subspace $W$ of $E(k_0,k_1,k_2,k_3)$ spanned by $v_{i-1}$ and $v_i$. 
By Proposition \ref{prop:E}(i), $W_i$ is $t_2$-invariant. Since $E(k_0,k_1,k_2,k_3)$ is a direct sum of $W_i$ for all $i=1,3,\ldots,d$ the lemma follows by evaluating the eigenspaces of $t_2$ in $W_i$.
\end{proof}

\begin{lem}\label{lem2:t0_E2}
\begin{enumerate}
\item If $k_2^2=1$ then $t_0$ is not diagonalizable on $E(k_0,k_1,k_2,k_3)^{2 \bmod{4}}$ with exactly one eigenvalue $k_2$.

\item If $k_2^2 \not=1$ then $t_0$ is diagonalizable on $E(k_0,k_1,k_2,k_3)^{2 \bmod{4}}$ with exactly two eigenvalues $k_2^{\pm 1}$. 
\end{enumerate}
\end{lem}
\begin{proof}
By Proposition \ref{prop:E}(ii) the element $t_2$ has at most two eigenvalues $k_2$ or $k_2^{-1}$ in the $\H_q$-module $E(k_0,k_1,k_2,k_3)$. Combined with  Lemma \ref{lem1:t0_E2} the lemma follows.
\end{proof}

\begin{lem}\label{lem:AB_E2(k2inv)}
Let 
$$
\mu_i=(-1)^{\frac{i-1}{2}} k_0^{-\frac{i-1}{2}} k_1^{-\frac{i-1}{2}} q^{-\frac{(i-1)(i+3)}{4}}
\qquad  
\hbox{for $i=1,3,\ldots,d$}.
$$
Then the matrices representing $A$ and $B$ with respect to the $\F$-basis 
\begin{gather}\label{basis_E2(k2inv)}
\mu_i((1-k_0 k_1 k_2 k_3 q^i ) v_{i-1}-v_i)
\quad 
\hbox{for $i=1,3,\ldots,d$}
\end{gather}
for the $\triangle_q$-module $E(k_0,k_1,k_2,k_3)^{2 \bmod{4}}(k_2^{-1})$ are 
$$
\begin{pmatrix}
\theta_0 & & &  &{\bf 0}
\\ 
1 &\theta_1 
\\
&1 &\theta_2 
 \\
& &\ddots &\ddots
 \\
{\bf 0} & & &1 &\theta_\frac{d-1}{2}
\end{pmatrix},
\qquad 
\begin{pmatrix}
\theta_0^* &\varphi_1 &  & &{\bf 0}
\\ 
 &\theta_1^* &\varphi_2
\\
 &  &\theta_2^* &\ddots
 \\
 & & &\ddots &\varphi_{\frac{d-1}{2}}
 \\
{\bf 0}  & & & &\theta_\frac{d-1}{2}^*
\end{pmatrix},
$$
respectively, where 
\begin{align*}
\theta_i &=
k_0 k_1 q^{2i+1}+k_0^{-1} k_1^{-1} q^{-2i-1}
\qquad 
\hbox{for $i=0,1,\ldots, \textstyle\frac{d-1}{2}$},
\\
\theta_i^* &= 
k_0 k_3 q^{2i+1}+k_0^{-1} k_3^{-1} q^{-2i-1}
\qquad 
\hbox{for $i=0,1,\ldots, \textstyle\frac{d-1}{2}$},
\\
\varphi_i &=
k_1^{-1} k_3^{-1} q^{\frac{d+1}{2}}
(q^i-q^{-i})
(q^{i-\frac{d+1}{2}}-q^{\frac{d+1}{2}-i})
\\
&\qquad \times \;
(q^{-i}-k_0 k_1 k_2 k_3 q^{i+1})
(q^{-i}-k_0 k_1 k_2^{-1} k_3 q^{i-1})
\qquad 
\hbox{for $i=1,2,\ldots, \textstyle\frac{d-1}{2}$}.
\end{align*}
The elements $\alpha,\beta,\gamma$ act on the $\triangle_q$-module $E(k_0,k_1,k_2,k_3)^{2 \bmod{4}}(k_2^{-1})$ as scalar multiplication by 
\begin{gather}
(k_1 +k_1^{-1})
(k_0+k_0^{-1})
+
(k_3+k_3^{-1})
(q k_2+q^{-1} k_2^{-1}),
\label{alpha:E2(k2inv)}
\\
(k_0+k_0^{-1})
(k_3+k_3^{-1})
+
(k_1 +k_1^{-1})
(q k_2+q^{-1} k_2^{-1}),
\label{beta:E2(k2inv)}
\\
(k_3+k_3^{-1})
(k_1 +k_1^{-1})
+
(k_0+k_0^{-1})
(q k_2+q^{-1} k_2^{-1}),
\label{gamma:E2(k2inv)}
\end{gather}
respectively.
\end{lem}
\begin{proof}
By Lemma \ref{lem1:t0_E2}(ii) the vectors (\ref{basis_E2(k2inv)}) are a basis for $E(k_0,k_1,k_2,k_3)^{2 \bmod{4}}(k_2^{-1})$. Applying Lemma \ref{lem:AB_E1}  a direct calculation yields the matrices representing $A$ and $B$ with respect to (\ref{basis_E2(k2inv)}). The actions of $c_0,c_1,c_2,c_3$ on $E(k_0,k_1,k_2,k_3)^{2 \bmod{4}}$ are identical to the actions of $c_2,c_3,c_0,c_1$ on $E(k_0,k_1,k_2,k_3)$. Combined with Theorem \ref{thm:hom} and Proposition \ref{prop:E}(ii) we get that $\alpha,\beta,\gamma$ act on $E(k_0,k_1,k_2,k_3)^{2 \bmod{4}}(k_2^{-1})$ as scalar multiplication by (\ref{alpha:E2(k2inv)})--(\ref{gamma:E2(k2inv)}), respectively.
\end{proof}

\begin{prop}\label{prop:E2(k2inv)}
The $\triangle_q$-module $E(k_0,k_1,k_2,k_3)^{2 \bmod{4}}(k_2^{-1})$ is isomorphic to 
$$
V_{\frac{d-1}{2}}\left(
k_0 k_1 q^{\frac{d+1}{2}},
k_0 k_3 q^{\frac{d+1}{2}},
k_0 k_2 q^{\frac{d+3}{2}}
\right).
$$
Moreover the $\triangle_q$-module $E(k_0,k_1,k_2,k_3)^{2 \bmod{4}}(k_2^{-1})$ is irreducible provided that the $\H_q$-module 
$E(k_0,k_1,k_2,k_3)$ is irreducible.
\end{prop}
\begin{proof}
Set $(a,b,c)=(k_0 k_1 q^{\frac{d+1}{2}},
k_0 k_3 q^{\frac{d+1}{2}},
k_0 k_2 q^{\frac{d+3}{2}})$
and $d'=\frac{d-1}{2}$. Under the assumption (\ref{k02=-d-1}) the scalar (\ref{beta:E2(k2inv)}) is equal to 
$$
(c+c^{-1})(a+a^{-1})+(b+b^{-1})(q^{d'+1}+q^{-d'-1}).
$$
By Proposition \ref{prop:UAWd} and Lemma \ref{lem:AB_E2(k2inv)} the  $\triangle_q$-module $E(k_0,k_1,k_2,k_3)^{2 \bmod{4}}(k_2^{-1})$ is isomorphic to $V_{d'}(a,b,c)$.

Suppose that the $\H_q$-module $E(k_0,k_1,k_2,k_3)$ is irreducible. Using Theorem \ref{thm:irr_E} yields that 
\begin{align*}
&abc^{-1}\not\in\{q^{2i-d'-1}\,|\,i=0,1,\ldots,d'\};
\\
&abc,a^{-1}bc,ab^{-1}c\not\in\{q^{2i-d'-1}\,|\,i=1,2,\ldots,d'+1\}.
\end{align*}
Hence the $\triangle_q$-module $V_{d'}(a,b,c)$ is irreducible by Theorem \ref{thm:irr_UAW}. The proposition follows.
\end{proof}

\begin{lem}\label{lem:AB_E2/E2(k2inv)}
Assume that the $\H_q$-module $E(k_0,k_1,k_2,k_3)$ is irreducible. 
Let 
\begin{gather}\label{mui}
\mu_i= \frac{(-1)^{\frac{i-1}{2}} k_0^{-\frac{i-1}{2}} k_1^{-\frac{i-1}{2}} q^{-\frac{(i-1)^2}{4}}}{1-k_0 k_1 k_2 k_3 q^i} 
\qquad 
\hbox{for $i=1,3,\ldots,d$}.
\end{gather}
Then the matrices representing $A$ and $B$ with respect to the $\F$-basis 
\begin{gather}\label{basis_E2/E2(k2inv)}
\mu_i v_i +E(k_0,k_1,k_2,k_3)^{2 \bmod{4}}(k_2^{-1})
\qquad \hbox{for $i=1,3,\ldots,d$}
\end{gather}
for the $\triangle_q$-module $E(k_0,k_1,k_2,k_3)^{2 \bmod{4}}/E(k_0,k_1,k_2,k_3)^{2 \bmod{4}}(k_2^{-1})$ are 
$$
\begin{pmatrix}
\theta_0 & & &  &{\bf 0}
\\ 
1 &\theta_1 
\\
&1 &\theta_2 
 \\
& &\ddots &\ddots
 \\
{\bf 0} & & &1 &\theta_\frac{d-1}{2}
\end{pmatrix},
\qquad 
\begin{pmatrix}
\theta_0^* &\varphi_1 &  & &{\bf 0}
\\ 
 &\theta_1^* &\varphi_2
\\
 &  &\theta_2^* &\ddots
 \\
 & & &\ddots &\varphi_{\frac{d-1}{2}}
 \\
{\bf 0}  & & & &\theta_\frac{d-1}{2}^*
\end{pmatrix},
$$
respectively, where 
\begin{align*}
\theta_i &=
k_0 k_1 q^{2i+1}+k_0^{-1} k_1^{-1} q^{-2i-1}
\qquad 
\hbox{for $i=0,1,\ldots, \textstyle\frac{d-1}{2}$},
\\
\theta_i^* &= 
k_0 k_3 q^{2i+1}+k_0^{-1} k_3^{-1} q^{-2i-1}
\qquad 
\hbox{for $i=0,1,\ldots, \textstyle\frac{d-1}{2}$},
\\
\varphi_i &=
k_1^{-1} k_3^{-1} q^{\frac{d+1}{2}}
(q^i-q^{-i})
(q^{i-\frac{d+1}{2}}-q^{\frac{d+1}{2}-i})
\\
&\qquad \times \;
(q^{-i}-k_0 k_1 k_2 k_3 q^{i-1})
(q^{-i}-k_0 k_1 k_2^{-1} k_3 q^{i+1})
\qquad 
\hbox{for $i=1,2,\ldots, \textstyle\frac{d-1}{2}$}.
\end{align*}
The elements $\alpha,\beta,\gamma$ act on the $\triangle_q$-module $E(k_0,k_1,k_2,k_3)^{2 \bmod{4}}/E(k_0,k_1,k_2,k_3)^{2 \bmod{4}}(k_2^{-1})$ as scalar multiplication by 
\begin{gather}
(k_1 +k_1^{-1})
(k_0+k_0^{-1})
+
(k_3+k_3^{-1})
(q k_2^{-1}+q^{-1} k_2),
\label{alpha:E2/E2(k2inv)}
\\
(k_0+k_0^{-1})
(k_3+k_3^{-1})
+
(k_1 +k_1^{-1})
(q k_2^{-1}+q^{-1} k_2),
\label{beta:E2/E2(k2inv)}
\\
(k_3+k_3^{-1})
(k_1 +k_1^{-1})
+
(k_0+k_0^{-1})
(q k_2^{-1}+q^{-1} k_2),
\label{gamma:E2/E2(k2inv)}
\end{gather}
respectively.
\end{lem}
\begin{proof}
By Theorem \ref{thm:irr_E} the denominator of (\ref{mui}) is nonzero. 
By Lemma \ref{lem1:t0_E2}(ii) the cosets (\ref{basis_E2/E2(k2inv)}) are a basis for the quotient of $E(k_0,k_1,k_2,k_3)^{2 \bmod{4}}$ modulo $E(k_0,k_1,k_2,k_3)^{2 \bmod{4}}(k_2^{-1})$. Applying Lemmas \ref{lem:AB_E2}  and \ref{lem1:t0_E2}(ii) it is routine to verify the matrices representing $A$ and $B$ with respect to (\ref{basis_E2/E2(k2inv)}). Using Proposition \ref{prop:E}(i) and Lemma \ref{lem1:t0_E2}(ii) it is routine to check that 
$$
(t_0-k_2)v_i\in E(k_0,k_1,k_2,k_3)^{2 \bmod{4}}(k_2^{-1})
\qquad 
\hbox{for $i=1,3,\ldots,d$}.
$$
Combined with Theorem \ref{thm:hom} and Proposition \ref{prop:E}(ii) the elements $\alpha,\beta,\gamma$ act on the quotient $\triangle_q$-module of  $E(k_0,k_1,k_2,k_3)^{2 \bmod{4}}$ modulo $E(k_0,k_1,k_2,k_3)^{2 \bmod{4}}(k_2^{-1})$ as scalar multiplication by (\ref{alpha:E2/E2(k2inv)})--(\ref{gamma:E2/E2(k2inv)}), respectively.
\end{proof}

\begin{prop}\label{prop:E2/E2(k2inv)}
Assume that the $\H_q$-module $E(k_0,k_1,k_2,k_3)$ is irreducible. 
The $\triangle_q$-module $E(k_0,k_1,k_2,k_3)^{2 \bmod{4}}/E(k_0,k_1,k_2,k_3)^{2 \bmod{4}}(k_2^{-1})$ is isomorphic to $$
V_{\frac{d-1}{2}}\left(
k_0 k_1 q^{\frac{d+1}{2}},
k_0 k_3 q^{\frac{d+1}{2}},
k_0 k_2 q^{\frac{d-1}{2}}
\right).
$$
Moreover the $\triangle_q$-module $E(k_0,k_1,k_2,k_3)^{2 \bmod{4}}/E(k_0,k_1,k_2,k_3)^{2 \bmod{4}}(k_2^{-1})$ is irreducible.
\end{prop}
\begin{proof}
Set $(a,b,c)=(
k_0 k_1 q^{\frac{d+1}{2}},
k_0 k_3 q^{\frac{d+1}{2}},
k_0 k_2 q^{\frac{d-1}{2}})$
and $d'=\frac{d-1}{2}$. Under the assumption (\ref{k02=-d-1}) the scalar (\ref{beta:E2/E2(k2inv)}) is equal to 
$$
(c+c^{-1})(a+a^{-1})+(b+b^{-1})(q^{d'+1}+q^{-d'-1}).
$$
By Proposition \ref{prop:UAWd} and Lemma \ref{lem:AB_E2/E2(k2inv)} the $\triangle_q$-module $E(k_0,k_1,k_2,k_3)^{2 \bmod{4}}$ is isomorphic to $V_{d'}(a,b,c)$.

Since the $\H_q$-module $E(k_0,k_1,k_2,k_3)$ is irreducible, it follows from Theorem \ref{thm:irr_E} that 
\begin{align*}
&abc^{-1}\not\in\{q^{2i-d'-1}\,|\,i=1,2,\ldots,d'+1\};
\\
&abc,a^{-1}bc,ab^{-1}c\not\in\{q^{2i-d'-1}\,|\,i=0,1,\ldots,d'\}.
\end{align*}
Hence the $\triangle_q$-module $V_{d'}(a,b,c)$ is irreducible by Theorem \ref{thm:irr_UAW}. The proposition follows.
\end{proof}

\begin{thm}\label{thm:E2}
Assume that the $\H_q$-module $E(k_0,k_1,k_2,k_3)$ is irreducible. Then the following hold:
\begin{enumerate}
\item If $k_2^2 =1$ then 
\begin{table}[H]
\begin{tikzpicture}[node distance=1.2cm]
 \node (0)                  {$\{0\}$};
 \node (E)  [above of=0]   {$E(k_0,k_1,k_2,k_3)^{2 \bmod{4}}(k_2)$};
 \node (V)  [above of=E]   {$E(k_0,k_1,k_2,k_3)^{2 \bmod{4}}$};
 \draw (0)   -- (E);
 \draw (E)  -- (V);
\end{tikzpicture}
\end{table}
\noindent is the lattice of $\triangle_q$-submodules of $E(k_0,k_1,k_2,k_3)^{2 \bmod{4}}$.

\item If $k_2^2 \not=1$ then 
 \begin{table}[H]
\begin{tikzpicture}[node distance=1.5cm]
 \node (0)                  {$\{0\}$};
 \node (1)  [above of=0]   {};
 \node (2)  [right of=1]   {};
 \node (3)  [left of=1]   {};
 \node (E1)  [right of=2]  {$E(k_0,k_1,k_2,k_3)^{2 \bmod{4}}(k_2)$};
 \node (E2)  [left of=3]   {$E(k_0,k_1,k_2,k_3)^{2 \bmod{4}}(k_2^{-1})$};
 \node (V) [above of=1]  {$E(k_0,k_1,k_2,k_3)^{2 \bmod{4}}$};
 \draw (0)   -- (E1);
 \draw (0)   -- (E2);
 \draw (E1)   -- (V);
 \draw (E2)  -- (V);
\end{tikzpicture}
\end{table}
\noindent is the lattice of $\triangle_q$-submodules of $E(k_0,k_1,k_2,k_3)^{2 \bmod{4}}$.
\end{enumerate}
\end{thm}
\begin{proof}
Using the above lemmas and propositions, the result follows by an  argument similar to the proof of Theorem \ref{thm:E1}.
\end{proof}

\subsection{The lattice of $\triangle_q$-submodules of $E(k_0,k_1,k_2,k_3)^{3 \bmod{4}}$}\label{s:lattice_E3}

Recall from Table \ref{Z/4Z-action} that $3\bmod{4}$ sends $t_0,t_1,t_2,t_3$ to $t_3,t_0,t_1,t_2$, respectively. 
In this subsection, we study the $\triangle_q$-submodules of the $\H_q$-module $E(k_0,k_1,k_2,k_3)^{3 \bmod{4}}$.

\begin{lem}\label{lem:AB_E3}
The actions of $A$ and $B$ on the $\H_q$-module $E(k_0,k_1,k_2,k_3)^{3 \bmod{4}}$ are as follows:
\begin{align*}
A v_i &=
\left\{
\begin{array}{ll}
\displaystyle
\theta_i^* v_i -k_0^{-1} k_3^{-1} q^{-i} \varrho_i \varrho_{i-1} v_{i-2}
\qquad 
&\hbox{for $i=2,4,\ldots,d-1$},
\\
\displaystyle
\theta_i^* v_i
+
k_0^{-1} k_3^{-1} q^{-i} (q-q^{-1}) \varrho_i  v_{i-1}
-
k_0^{-1} k_3^{-1} q^{1-i} \varrho_i \varrho_{i-1} v_{i-2}
\qquad 
&\hbox{for $i=3,5,\ldots,d$},
\end{array}
\right.
\\
A v_0 &=
\theta_0^* v_0,
\qquad 
A v_1 =
\theta_1^* v_1+
k_0^{-1} k_3^{-1} q^{-1} (q-q^{-1}) \varrho_1  v_0,
\\
B v_i &=
\left\{
\begin{array}{ll}
\displaystyle
\theta_i v_i 
-
k_0^{-1} k_1^{-1} q^{-i-1} v_{i+2}
\qquad 
&\hbox{for $i=0,2,\ldots,d-3$},
\\
\displaystyle
\theta_i v_i
-
k_0^{-1} k_1^{-1} q^{-i-1}(q-q^{-1}) v_{i+1}
-
k_0^{-1} k_1^{-1} q^{-i-2} v_{i+2}
\qquad 
&\hbox{for $i=3,5,\ldots,d-2$},
\end{array}
\right.
\\
B v_{d-1}&=\theta_{d-1} v_{d-1},
\qquad 
B v_d = \theta_d v_d,
\end{align*}
where 
\begin{align*}
\theta_i
&=
k_0 k_1 q^{2\lfloor \frac{i}{2}\rfloor +1}
+k_0^{-1} k_1^{-1}  q^{-2\lfloor \frac{i}{2}\rfloor -1}
\qquad 
\hbox{for $i=0,1,\ldots,d$},
\\
\theta_i^*
&=
k_0 k_3 q^{2\lceil \frac{i}{2}\rceil}
+k_0^{-1} k_3^{-1}  q^{-2\lceil \frac{i}{2}\rceil}
\qquad 
\hbox{for $i=0,1,\ldots,d$}.
\end{align*}
\end{lem}
\begin{proof}
By Lemma \ref{lem:X+Xinv} the actions of $A$ and $B$ on $E(k_0,k_1,k_2,k_3)^{3 \bmod{4}}$ are identical to the actions of $X+X^{-1}$ and 
$q Y+q^{-1} Y^{-1}$ 
on $E(k_0,k_1,k_2,k_3)$, respectively. Using Lemma \ref{lem:XYinE} a direct calculation yields the actions of $X+X^{-1}$ and 
$q Y+q^{-1} Y^{-1}$  on $E(k_0,k_1,k_2,k_3)$.
\end{proof}

\begin{lem}\label{lem1:t0_E3}
Assume that the $\H_q$-module $E(k_0,k_1,k_2,k_3)$ is irreducible. Then the matrix representing $t_0$ with respect to the $\F$-basis 
\begin{gather}\label{basis:E3}
\varrho_i v_{i-1}-k_3^2 v_{i+1}
\quad 
\hbox{for $i=1,3,\ldots,d-2$},
\quad 
\varrho_d v_{d-1},
\quad 
v_i 
\quad 
\hbox{for $i=1,3,\ldots,d$}
\end{gather}
for $E(k_0,k_1,k_2,k_3)^{3 \bmod{4}}$ is 
\begin{gather*}
\begin{pmatrix}
k_3 I_{\frac{d+1}{2}} & \rvline & -k_3^{-1} I_{\frac{d+1}{2}}
\\
\hline
{\bf 0} & \rvline & k_3^{-1}  I_{\frac{d+1}{2}} 
\end{pmatrix}.
\end{gather*}
\end{lem}
\begin{proof}
Note that the action of $t_0$ on $E(k_0,k_1,k_2,k_3)^{3 \bmod{4}}$ is identical to the action of $t_3$ on $E(k_0,k_1,k_2,k_3)$. It is routine to verify the lemma using Proposition \ref{prop:E}(i).
\end{proof}

\begin{lem}\label{lem2:t0_E3}
Assume that the $\H_q$-module $E(k_0,k_1,k_2,k_3)$ is irreducible.  Then the following hold:
\begin{enumerate}
\item If $k_3^2 =1$ then $t_0$ is not diagonalizable on $E(k_0,k_1,k_2,k_3)^{3 \bmod{4}}$ with exactly one eigenvalue $k_3$.

\item If $k_3^2 \not=1$ then $t_0$ is diagonalizable on $E(k_0,k_1,k_2,k_3)^{3 \bmod{4}}$ with exactly two eigenvalues $k_3^{\pm 1}$. 
\end{enumerate}
\end{lem}
\begin{proof}
Applying the rank-nullity theorem to Lemma \ref{lem1:t0_E3} the lemma follows.
\end{proof}

\begin{lem}\label{lem3:t0_E3}
If the $\H_q$-module $E(k_0,k_1,k_2,k_3)$ is irreducible then 
$E(k_0,k_1,k_2,k_3)^{3 \bmod{4}}(k_3)$ is of dimension $\frac{d+1}{2}$ with the $\F$-basis 
\begin{gather}\label{basis:E3(k3)}
v_i
\qquad 
\hbox{for $i=0,2,\ldots,d-1$}.
\end{gather}
\end{lem}
\begin{proof}
By Lemma \ref{lem1:t0_E3} the vectors 
$\varrho_i v_{i-1}-k_3^2 v_{i+1}$ for $i=1,3,\ldots,d-2$ and $\varrho_d v_{d-1}$ form an $\F$-basis for $E(k_0,k_1,k_2,k_3)^{1 \bmod{4}}(k_1^{-1})$ as well as (\ref{basis:E3(k3)}).
\end{proof}

\begin{lem}\label{lem:AB_E3(k3)}
Assume that the $\H_q$-module $E(k_0,k_1,k_2,k_3)$ is irreducible. Let 
$$
\mu_i=(-1)^{\frac{i}{2}} k_0^{-\frac{i}{2}} k_1^{-\frac{i}{2}} q^{-\frac{i^2}{4}}
\qquad  
\hbox{for $i=0,2,\ldots,d-1$}.
$$
Then the matrices representing $A$ and $B$ with respect to the $\F$-basis 
\begin{gather}\label{basis_E3(k3)}
\mu_i v_i
\qquad \hbox{for $i=0,2,\ldots,d-1$}
\end{gather}
for the $\triangle_q$-module $E(k_0,k_1,k_2,k_3)^{3 \bmod{4}}(k_3)$ are 
$$
\begin{pmatrix}
\theta_0^* &\varphi_1 &  & &{\bf 0}
\\ 
 &\theta_1^* &\varphi_2
\\
 &  &\theta_2^* &\ddots
 \\
 & & &\ddots &\varphi_{\frac{d-1}{2}}
 \\
{\bf 0}  & & & &\theta_\frac{d-1}{2}^*
\end{pmatrix},
\qquad 
\begin{pmatrix}
\theta_0 & & &  &{\bf 0}
\\ 
1 &\theta_1 
\\
&1 &\theta_2 
 \\
& &\ddots &\ddots
 \\
{\bf 0} & & &1 &\theta_\frac{d-1}{2}
\end{pmatrix},
$$
respectively, where   
\begin{align*}
\theta_i &=
k_0 k_1 q^{2i+1}+k_0^{-1} k_1^{-1} q^{-2i-1}
\qquad 
\hbox{for $i=0,1,\ldots, \textstyle\frac{d-1}{2}$},
\\
\theta_i^* &= 
k_0 k_3 q^{2i}+k_0^{-1} k_3^{-1} q^{-2i}
\qquad 
\hbox{for $i=0,1,\ldots, \textstyle\frac{d-1}{2}$},
\\
\varphi_i &=
k_1^{-1} k_3^{-1} q^{\frac{d+3}{2}}
(q^i-q^{-i})
(q^{i-\frac{d+1}{2}}-q^{\frac{d+1}{2}-i})
\\
&\qquad \times \;
(q^{-i}-k_0 k_1 k_2 k_3 q^{i-1})
(q^{-i}-k_0 k_1 k_2^{-1} k_3 q^{i-1})
\qquad 
\hbox{for $i=1,2,\ldots, \textstyle\frac{d-1}{2}$}.
\end{align*}
The elements $\alpha,\beta,\gamma$ act on the $\triangle_q$-module $E(k_0,k_1,k_2,k_3)^{3 \bmod{4}}(k_3)$ as scalar multiplication by 
\begin{gather}
(k_2+k_2^{-1})
(k_1+k_1^{-1})
+
(k_0+k_0^{-1})
(q k_3^{-1}+q^{-1} k_3),
\label{alpha:E3(k3)}
\\
(k_1+k_1^{-1})
(k_0+k_0^{-1})
+
(k_2+k_2^{-1})
(q k_3^{-1}+q^{-1} k_3),
\label{beta:E3(k3)}
\\
(k_0+k_0^{-1})
(k_2+k_2^{-1})
+
(k_1+k_1^{-1})
(q k_3^{-1}+q^{-1} k_3),
\label{gamma:E3(k3)}
\end{gather}
respectively.
\end{lem}
\begin{proof}
By Lemma \ref{lem3:t0_E3} the vectors (\ref{basis_E3(k3)}) are an $\F$-basis for $E(k_0,k_1,k_2,k_3)^{3 \bmod{4}}(k_3)$. By Lemma \ref{lem:AB_E3}  we obtain the matrices representing $A$ and $B$ with respect to (\ref{basis_E3(k3)}). The actions of $c_0,c_1,c_2,c_3$ on $E(k_0,k_1,k_2,k_3)^{3 \bmod{4}}$ are identical to the actions of $c_3,c_0,c_1,c_2$ on $E(k_0,k_1,k_2,k_3)$. Combined with Theorem \ref{thm:hom} and Proposition \ref{prop:E}(ii) the elements $\alpha,\beta,\gamma$ act on $E(k_0,k_1,k_2,k_3)^{3 \bmod{4}}(k_3)$ as scalar multiplication by (\ref{alpha:E3(k3)})--(\ref{gamma:E3(k3)}), respectively.
\end{proof}

\begin{prop}\label{prop:E3(k3)}
Assume that the $\H_q$-module $E(k_0,k_1,k_2,k_3)$ is irreducible. 
Then the $\triangle_q$-module $E(k_0,k_1,k_2,k_3)^{3 \bmod{4}}(k_3)$ is isomorphic to $$
V_{\frac{d-1}{2}}\left(
k_0 k_3 q^{\frac{d-1}{2}},
k_0 k_1 q^{\frac{d+1}{2}},
k_0 k_2 q^{\frac{d+1}{2}}
\right).
$$
Moreover the $\triangle_q$-module $E(k_0,k_1,k_2,k_3)^{3 \bmod{4}}(k_3)$ is irreducible.
\end{prop}
\begin{proof}
Set $(a,b,c)=(
k_0 k_1 q^{\frac{d+1}{2}},
k_0 k_3 q^{\frac{d-1}{2}},
k_0 k_2 q^{\frac{d+1}{2}})$
and $d'=\frac{d-1}{2}$. Under the assumption (\ref{k02=-d-1}) the scalar (\ref{beta:E3(k3)}) is equal to 
$$
(b+b^{-1})(c+c^{-1})+(a+a^{-1})(q^{d'+1}+q^{-d'-1}).
$$
By Proposition \ref{prop:UAWd} and Lemma \ref{lem:AB_E3(k3)} the $\triangle_q$-module $E(k_0,k_1,k_2,k_3)^{3 \bmod{4}}(k_3)$ is isomorphic to $V_{d'}(a,b,c)^{1 \bmod{2}}$.

Since the $\H_q$-module $E(k_0,k_1,k_2,k_3)$ is irreducible it follows from Theorem \ref{thm:irr_E} that 
\begin{align*}
&ab^{-1}c\not\in\{q^{2i-d'-1}\,|\,i=1,2,\ldots,d'+1\};
\\
&abc,a^{-1}bc,abc^{-1}\not\in\{q^{2i-d'-1}\,|\,i=0,1,\ldots,d'\}.
\end{align*}
Hence the $\triangle_q$-module $V_{d'}(a,b,c)$ is irreducible by Theorem \ref{thm:irr_UAW}. 
By Lemma \ref{lem:Z/2Z} the $\triangle_q$-module $V_{d'}(a,b,c)^{1 \bmod{2}}$ is isomorphic to $V_{d'}(b,a,c)$. 
The proposition follows.
\end{proof}

\begin{lem}\label{lem:AB_E3/E3(k3)}
Assume that the $\H_q$-module $E(k_0,k_1,k_2,k_3)$ is irreducible.  
Let 
$$
\mu_i=(-1)^{\frac{i-1}{2}} k_0^{-\frac{i-1}{2}} k_1^{-\frac{i-1}{2}} q^{-\frac{(i-1)(i+3)}{4}} 
\qquad 
\hbox{for $i=1,3,\ldots,d$}.
$$
Then the matrices representing $A$ and $B$ with respect to the $\F$-basis 
\begin{gather}\label{basis_E3/E3(k3)}
\mu_i v_i+E(k_0,k_1,k_2,k_3)^{3 \bmod{4}}(k_3)
\qquad \hbox{for $i=1,3,\ldots,d$}
\end{gather}
for the $\triangle_q$-module $E(k_0,k_1,k_2,k_3)^{3 \bmod{4}}/E(k_0,k_1,k_2,k_3)^{3 \bmod{4}}(k_3)$ are 
$$
\begin{pmatrix}
\theta_0^* &\varphi_1 &  & &{\bf 0}
\\ 
 &\theta_1^* &\varphi_2
\\
 &  &\theta_2^* &\ddots
 \\
 & & &\ddots &\varphi_{\frac{d-1}{2}}
 \\
{\bf 0}  & & & &\theta_\frac{d-1}{2}^*
\end{pmatrix},
\qquad 
\begin{pmatrix}
\theta_0 & & &  &{\bf 0}
\\ 
1 &\theta_1 
\\
&1 &\theta_2
 \\
& &\ddots &\ddots
 \\
{\bf 0} & & &1 &\theta_\frac{d-1}{2}
\end{pmatrix},
$$
respectively, where  
\begin{align*}
\theta_i &=
k_0 k_1 q^{2i+1}+k_0^{-1} k_1^{-1} q^{-2i-1}
\qquad 
\hbox{for $i=0,1,\ldots, \textstyle\frac{d-1}{2}$},
\\
\theta_i^* &= 
k_0 k_3 q^{2i+2}+k_0^{-1} k_3^{-1} q^{-2i-2}
\qquad 
\hbox{for $i=0,1,\ldots, \textstyle\frac{d-1}{2}$},
\\
\varphi_i &=
k_0^{-1} k_3^{-1} q^{\frac{d-1}{2}}
(q^i-q^{-i})
(q^{i-\frac{d+1}{2}}-q^{\frac{d+1}{2}-i})
\\
&\qquad \times \;
(q^{-i}-k_0 k_1 k_2 k_3 q^{i+1})
(q^{-i}-k_0 k_1 k_2^{-1} k_3 q^{i+1})
\qquad 
\hbox{for $i=1,2,\ldots, \textstyle\frac{d-1}{2}$}.
\end{align*}
The elements $\alpha,\beta,\gamma$ act on the $\triangle_q$-module $E(k_0,k_1,k_2,k_3)^{3 \bmod{4}}/E(k_0,k_1,k_2,k_3)^{3 \bmod{4}}(k_3)$ as scalar multiplication by 
\begin{gather}
(k_2+k_2^{-1})
(k_1+k_1^{-1})
+
(k_0+k_0^{-1})
(q k_3+q^{-1} k_3^{-1}),
\label{alpha:E3/E3(k3)}
\\
(k_1+k_1^{-1})
(k_0+k_0^{-1})
+
(k_2+k_2^{-1})
(q k_3+q^{-1} k_3^{-1}),
\label{beta:E3/E3(k3)}
\\
(k_0+k_0^{-1})
(k_2+k_2^{-1})
+
(k_1+k_1^{-1})
(q k_3+q^{-1} k_3^{-1}),
\label{gamma:E3/E3(k3)}
\end{gather}
respectively.
\end{lem}
\begin{proof}
By Lemma \ref{lem3:t0_E3} the cosets (\ref{basis_E3/E3(k3)}) are a basis for the quotient of $E(k_0,k_1,k_2,k_3)^{3 \bmod{4}}$ modulo $E(k_0,k_1,k_2,k_3)^{3 \bmod{4}}(k_3)$. By Lemmas \ref{lem:AB_E3} and \ref{lem3:t0_E3} we obtain the matrices representing $A$ and $B$ with respect to (\ref{basis_E3/E3(k3)}). By Lemma \ref{lem1:t0_E3} we have 
\begin{gather*}
(t_0-k_3^{-1}) v_i\in  E(k_0,k_1,k_2,k_3)^{3 \bmod{4}}(k_3)
\qquad 
\hbox{for $i=1,3,\ldots,d$}.
\end{gather*}
Combined with Theorem \ref{thm:hom} and Proposition \ref{prop:E}(ii) we see that $\alpha,\beta,\gamma$ act on the quotient $\triangle_q$-module of $E(k_0,k_1,k_2,k_3)^{3 \bmod{4}}$ modulo $E(k_0,k_1,k_2,k_3)^{3 \bmod{4}}(k_3)$ as scalar multiplication by (\ref{alpha:E3/E3(k3)})--(\ref{gamma:E3/E3(k3)}), respectively.
\end{proof}

\begin{prop}\label{prop:E3/E3(k3)}
Assume that the $\H_q$-module $E(k_0,k_1,k_2,k_3)$ is irreducible. 
Then the $\triangle_q$-module $E(k_0,k_1,k_2,k_3)^{3 \bmod{4}}/E(k_0,k_1,k_2,k_3)^{3 \bmod{4}}(k_3)$ is isomorphic to 
$$
V_{\frac{d-1}{2}}\left(
k_0 k_3 q^{\frac{d+3}{2}},
k_0 k_1 q^{\frac{d+1}{2}},
k_0 k_2 q^{\frac{d+1}{2}}
\right).
$$
Moreover the $\triangle_q$-module 
$
E(k_0,k_1,k_2,k_3)^{3 \bmod{4}}/E(k_0,k_1,k_2,k_3)^{3 \bmod{4}}(k_3)
$ 
is irreducible.
\end{prop}
\begin{proof}
Set $(a,b,c)=(
k_0 k_1 q^{\frac{d+1}{2}},
k_0 k_3 q^{\frac{d+3}{2}},
k_0 k_2 q^{\frac{d+1}{2}}
)$
and $d'=\frac{d-1}{2}$. Under the assumption (\ref{k02=-d-1}) the scalar (\ref{beta:E3/E3(k3)}) is equal to 
$$
(b+b^{-1})(c+c^{-1})+(a+a^{-1})(q^{d'+1}+q^{-d'-1}).
$$
By Proposition \ref{prop:UAWd} and Lemma \ref{lem:AB_E3/E3(k3)} the quotient $\triangle_q$-module $E(k_0,k_1,k_2,k_3)^{3 \bmod{4}}$ modulo $E(k_0,k_1,k_2,k_3)^{3 \bmod{4}}(k_3)$ is isomorphic to $V_{d'}(a,b,c)^{1\bmod{2}}$.

Since the $\H_q$-module $E(k_0,k_1,k_2,k_3)$ is irreducible, it follows from Theorem \ref{thm:irr_E} that 
\begin{align*}
& ab^{-1}c\not\in\{q^{2i-d'-1}\,|\,i=0,1,\ldots,d'\};
\\
& abc,a^{-1}bc,abc^{-1}\not\in\{q^{2i-d'-1}\,|\,i=1,2,\ldots,d'+1\}.
\end{align*}
Hence the $\triangle_q$-module $V_{d'}(a,b,c)$ is irreducible by Theorem \ref{thm:irr_UAW}. 
By Lemma \ref{lem:Z/2Z} the $\triangle_q$-module $V_{d'}(a,b,c)^{1\bmod{2}}$ is isomorphic to $V_{d'}(b,a,c)$.
The proposition follows.
\end{proof}

\begin{thm}\label{thm:E3}
Assume that the $\H_q$-module $E(k_0,k_1,k_2,k_3)$ is irreducible. Then the following hold:
\begin{enumerate}
\item If $k_3^2=1$ then 
\begin{table}[H]
\begin{tikzpicture}[node distance=1.2cm]
 \node (0)                  {$\{0\}$};
 \node (E)  [above of=0]   {$E(k_0,k_1,k_2,k_3)^{3 \bmod{4}}(k_3)$};
 \node (V)  [above of=E]   {$E(k_0,k_1,k_2,k_3)^{3 \bmod{4}}$};
 \draw (0)   -- (E);
 \draw (E)  -- (V);
\end{tikzpicture}
\end{table}
\noindent is the lattice of $\triangle_q$-submodules of $E(k_0,k_1,k_2,k_3)^{3 \bmod{4}}$.

\item If $k_3^2\not=1$ then 
 \begin{table}[H]
\begin{tikzpicture}[node distance=1.5cm]
 \node (0)                  {$\{0\}$};
 \node (1)  [above of=0]   {};
 \node (2)  [right of=1]   {};
 \node (3)  [left of=1]   {};
 \node (E1)  [right of=2]  {$E(k_0,k_1,k_2,k_3)^{3 \bmod{4}}(k_3)$};
 \node (E2)  [left of=3]   {$E(k_0,k_1,k_2,k_3)^{3 \bmod{4}}(k_3^{-1})$};
 \node (V) [above of=1]  {$E(k_0,k_1,k_2,k_3)^{3 \bmod{4}}$};
 \draw (0)   -- (E1);
 \draw (0)   -- (E2);
 \draw (E1)   -- (V);
 \draw (E2)  -- (V);
\end{tikzpicture}
\end{table}
\noindent is the lattice of $\triangle_q$-submodules of $E(k_0,k_1,k_2,k_3)^{3 \bmod{4}}$.
\end{enumerate}
\end{thm}
\begin{proof}
Using the above lemmas and propositions, the result follows by an  argument similar to the proof of Theorem \ref{thm:E1}.
\end{proof}

\subsection{The lattice of $\triangle_q$-submodules of $O(k_0,k_1,k_2,k_3)$}\label{s:lattice_O}

Throughout this subsection  we use the following conventions: Let $d\geq 0$ denote an even integer and assume that $k_0,k_1,k_2,k_3$ are nonzero scalars in $\F$ with 
\begin{gather}\label{k0123=-d-1}
k_0 k_1 k_2 k_3=q^{-d-1}.
\end{gather} 
Let $\{v_i\}_{i=0}^d$ denote the basis for $O(k_0,k_1,k_2,k_3)$ from Proposition \ref{prop:O}(i). Set 
$$
\varrho_i
=\left\{
\begin{array}{ll}
(q^{i-d-1}-1)(k_2^{-2} q^{i-d-1}-1)
\qquad 
&\hbox{for $i=1,3,\ldots,d-1$},
\\
(1-q^i) (1-k_0^2 q^i)
\qquad 
&\hbox{for $i=2,4,\ldots,d$}.
\end{array}
\right.
$$

In this subsection, we investigate the $\triangle_q$-submodules of the $\H_q$-module $O(k_0,k_1,k_2,k_3)$.

\begin{lem}\label{lem:AB_O}
The actions of $A$ and $B$ on the $\H_q$-module $O(k_0,k_1,k_2,k_3)$ are as follows:
\begin{align*}
A v_i &=
\left\{
\begin{array}{ll}
\displaystyle
\theta_i v_i 
-
k_0^{-1} k_1^{-1} q^{-i-1}(q-q^{-1}) v_{i+1}
-
k_0^{-1} k_1^{-1} q^{-i-2} v_{i+2}
\qquad 
&\hbox{for $i=0,2,\ldots,d-2$},
\\
\displaystyle
\theta_i v_i -k_0^{-1} k_1^{-1} q^{-i-1} v_{i+2}
\qquad 
&\hbox{for $i=1,3,\ldots,d-3$},
\end{array}
\right.
\\
A v_{d-1}&=\theta_{d-1} v_{d-1},
\qquad 
A v_d = \theta_d v_d,
\\
B v_i &=
\left\{
\begin{array}{ll}
\displaystyle
\theta_i^* v_i -k_0^{-1} k_3^{-1} q^{-i} \varrho_i \varrho_{i-1} v_{i-2}
\qquad 
&\hbox{for $i=2,4,\ldots,d$},
\\
\displaystyle
\theta_i^* v_i
+
k_0^{-1} k_3^{-1} q^{-i} (q-q^{-1}) \varrho_i  v_{i-1}
-
k_0^{-1} k_3^{-1} q^{1-i} \varrho_i \varrho_{i-1} v_{i-2}
\qquad 
&\hbox{for $i=3,5,\ldots,d-1$},
\end{array}
\right.
\\
B v_0 &=
\theta^*_0 v_0,
\qquad 
B v_1 =
\theta^*_1 v_1+
k_0^{-1} k_3^{-1} q^{-1} (q-q^{-1}) \varrho_1  v_0,
\end{align*}
where 
\begin{align*}
\theta_i
&=
k_0 k_1 q^{2\lceil \frac{i}{2}\rceil}
+k_0^{-1} k_1^{-1}  q^{-2\lceil \frac{i}{2}\rceil}
\qquad 
\hbox{for $i=0,1,\ldots,d$},
\\
\theta_i^*
&=
k_0 k_3 q^{2\lceil \frac{i}{2}\rceil}
+k_0^{-1} k_3^{-1}  q^{-2\lceil \frac{i}{2}\rceil}
\qquad 
\hbox{for $i=0,1,\ldots,d$}.
\end{align*}
\end{lem}
\begin{proof}
By Lemma \ref{lem:X+Xinv} we have $A=Y+Y^{-1}$ and $B=X+X^{-1}$. Using Lemma \ref{lem:XYinO} it is routine to evaluate the actions of $Y+Y^{-1}$ and $X+X^{-1}$ on $O(k_0,k_1,k_2,k_3)$.
\end{proof}

\begin{lem}\label{lem1:t0_O}
The matrix representing $t_0$ with respect to the $\F$-basis 
\begin{gather*}
v_0,
\quad 
v_{i-1}+q^{-i} (v_i-v_{i-1})
\quad \hbox{for $i=2,4,\ldots,d$},
\quad 
v_i 
\quad \hbox{for $i=1,3,\ldots,d-1$}
\end{gather*}
for $O(k_0,k_1,k_2,k_3)$ is 
\begin{gather*}
\begin{pmatrix}
k_0   &\rvline &{\bf 0}  &\rvline &{\bf 0} 
\\
\hline
{\bf 0}  &\rvline
&k_0 I_{\frac{d}{2}}
 &\rvline & -k_0^{-1} I_{\frac{d}{2}}
\\
\hline
{\bf 0} &  \rvline 
&{\bf 0} &  \rvline
&k_0^{-1} I_{\frac{d}{2}} 
\end{pmatrix}.
\end{gather*}
\end{lem}
\begin{proof}
Apply Proposition \ref{prop:O}(i) to verify the lemma.
\end{proof}

\begin{lem}\label{lem2:t0_O}
\begin{enumerate}
\item If $d=0$ then $t_0$ is diagonalizable on $O(k_0,k_1,k_2,k_3)$ with exactly one eigenvalue $k_0$.  

\item If $d\geq 2$ and $k_0^2 =1$ then $t_0$ is not diagonalizable on $O(k_0,k_1,k_2,k_3)$ with exactly one eigenvalue $k_0$.

\item If $d\geq 2$ and $k_0^2\not= 1$ then $t_0$ is diagonalizable on $O(k_0,k_1,k_2,k_3)$ with exactly two eigenvalues 
$k_0^{\pm1}$. 
\end{enumerate}
\end{lem}
\begin{proof}
To see the lemma, apply the rank-nullity theorem to Lemma \ref{lem1:t0_O}. 
\end{proof}

\begin{lem}\label{lem3:t0_O}
$O(k_0,k_1,k_2,k_3)(
k_0)$ is of dimension $\frac{d}{2}+1$ with the $\F$-basis
$$
v_0,
\quad
v_{i-1}+q^{-i} (v_i-v_{i-1})
\quad \hbox{for $i=2,4,\ldots,d$}.
$$
\end{lem}
\begin{proof}
Immediate from Lemma \ref{lem1:t0_O}.
\end{proof}

\begin{lem}\label{lem:AB_O(k0)}
Let 
$$
\mu_i=(-1)^{\frac{i}{2}}k_0^{-\frac{i}{2}}  
k_1^{-\frac{i}{2}} 
q^{-\frac{i(i-2)}{4}}
\qquad \hbox{for $i=2,4,\ldots,d$}.
$$
Then the matrices representing $A$ and $B$ with respect to the $\F$-basis 
\begin{gather}\label{basis:O(k0)}
v_0,
\quad 
\mu_i
(v_{i-1}+q^{-i} (v_i-v_{i-1}))
\quad \hbox{for $i=2,4,\ldots,d$}
\end{gather}
for the $\triangle_q$-module $O(k_0,k_1,k_2,k_3)(k_0)$ are 
$$
\begin{pmatrix}
\theta_0 & & &  &{\bf 0}
\\ 
1 &\theta_1 
\\
&1 &\theta_2
 \\
& &\ddots &\ddots
 \\
{\bf 0} & & &1 &\theta_\frac{d}{2}
\end{pmatrix},
\qquad 
\begin{pmatrix}
\theta_0^* &\varphi_1 &  & &{\bf 0}
\\ 
 &\theta_1^* &\varphi_2
\\
 &  &\theta_2^* &\ddots
 \\
 & & &\ddots &\varphi_{\frac{d}{2}}
 \\
{\bf 0}  & & & &\theta_\frac{d}{2}^*
\end{pmatrix},
$$
respectively, where 
\begin{align*}
\theta_i &=
k_0 k_1 q^{2i}
+k_0^{-1} k_1^{-1} q^{-2i}
\qquad 
\hbox{for $i=0,1,\ldots, \textstyle\frac{d}{2}$},
\\
\theta_i^* &=
k_0 k_3 q^{2i}
+k_0^{-1} k_3^{-1} q^{-2i}
\qquad 
\hbox{for $i=0,1,\ldots, \textstyle\frac{d}{2}$},
\\
\varphi_i &=
k_0^{-1} k_2 q^{\frac{d}{2}+2}
(q^i-q^{-i}) (q^{i-\frac{d}{2}-1}-q^{\frac{d}{2}-i+1})
\\
&\qquad \times \;
(q^{-i}-k_0^2 q^{i-2}) (q^{-i}-k_2^{-2} q^{i-d-2})
\qquad 
\hbox{for $i=1,2,\ldots, \textstyle\frac{d}{2}$}.
\end{align*}
The elements $\alpha,\beta,\gamma$ act on the $\triangle_q$-module $O(k_0,k_1,k_2,k_3)(k_0)$ as scalar multiplication by 
\begin{gather}
(k_3+k_3^{-1})(k_2+k_2^{-1})+(k_1+k_1^{-1})(q^{-1} k_0+q k_0^{-1}),
\label{alpha:O(k0)}
\\
(k_2+k_2^{-1})(k_1+k_1^{-1})+(k_3+k_3^{-1})(q^{-1} k_0+q k_0^{-1}),
\label{beta:O(k0)}
\\
(k_1+k_1^{-1})(k_3+k_3^{-1})+(k_2+k_2^{-1})(q^{-1} k_0+q k_0^{-1}),
\label{gam:O(k0)}
\end{gather}
respectively.
\end{lem}
\begin{proof}
By Lemma \ref{lem3:t0_O} the vectors (\ref{basis:O(k0)}) are an $\F$-basis for $E(k_0,k_1,k_2,k_3)(k_0)$. Applying Lemma \ref{lem:AB_O} a direct calculation yields the matrices representing 
$A$ and $B$ with respect to (\ref{basis:O(k0)}). By Theorem \ref{thm:hom} and Proposition \ref{prop:O}(ii) the elements $\alpha,\beta,\gamma$ act on $O(k_0,k_1,k_2,k_3)(k_0)$ as scalar multiplication by (\ref{alpha:O(k0)})--(\ref{gam:O(k0)}), respectively. 
\end{proof}

\begin{prop}\label{prop:O(k0)}
The $\triangle_q$-module $O(k_0,k_1,k_2,k_3)(k_0)$ is isomorphic to 
$$
V_{\frac{d}{2}}\left(
k_0 k_1 q^{\frac{d}{2}},
k_0 k_3 q^{\frac{d}{2}},
k_0 k_2 q^{\frac{d}{2}}
\right).
$$
Moreover the $\triangle_q$-module $O(k_0,k_1,k_2,k_3)(k_0)$ is irreducible if $k_0^2\not= 1$ and the $\H_q$-module $O(k_0,k_1,k_2,k_3)$ is irreducible.
\end{prop}
\begin{proof}
Set $(a,b,c)=(
k_0 k_1 q^{\frac{d}{2}},
k_0 k_3 q^{\frac{d}{2}},
k_0 k_2 q^{\frac{d}{2}})$ and $d'=\frac{d}{2}$. Under the assumption (\ref{k0123=-d-1}) the scalar (\ref{beta:O(k0)}) is equal to 
$$
(c+c^{-1})(a+a^{-1})+(b+b^{-1})(q^{d'+1}+q^{-d'-1}).
$$
Comparing Proposition \ref{prop:UAWd} with Lemma \ref{lem:AB_O(k0)} we see that the $\triangle_q$-module  $O(k_0,k_1,k_2,k_3)(k_0)$ is isomorphic to $V_{d'}(a,b,c)$. 

Suppose that $k_0^2\not=1$ and the $\H_q$-module $O(k_0,k_1,k_2,k_3)$ is irreducible. Combined with Theorem \ref{thm:irr_O} this yields that 
\begin{align*}
&abc\not\in\{q^{2i-d'-1}\,|\,i=0,1,\ldots,d'\};
\\
&a^{-1}bc,ab^{-1}c,abc^{-1}\not\in\{q^{2i-d'-1}\,|\,i=1,2,\ldots,d'\}.
\end{align*}
Hence the $\triangle_q$-module $V_{d'}(a,b,c)$ is irreducible by Theorem \ref{thm:irr_UAW}. The proposition follows.
\end{proof}

\begin{lem}\label{lem:AB_O/O(k0)}
Suppose that $d\geq 2$. Let
$$
\mu_i=(-1)^{\frac{i-1}{2}} k_0^{-\frac{i-1}{2}} k_1^{-\frac{i-1}{2}} q^{-\frac{(i-1)(i+1)}{4}}
\qquad 
\hbox{for $i=1,3,\ldots,d-1$}.
$$
Then the matrices representing $A$ and $B$ with respect to the $\F$-basis 
\begin{gather}\label{basis:O/O(k0)}
\mu_i v_i+O(k_0,k_1,k_2,k_3)(k_0)
\qquad 
\hbox{for $i=1,3,\ldots,d-1$}
\end{gather}
for the $\triangle_q$-module $O(k_0,k_1,k_2,k_3)/O(k_0,k_1,k_2,k_3)(k_0)$ are 
$$
\begin{pmatrix}
\theta_0 & & &  &{\bf 0}
\\ 
1 &\theta_1 
\\
&1 &\theta_2
 \\
& &\ddots &\ddots
 \\
{\bf 0} & & &1 &\theta_{\frac{d}{2}-1}
\end{pmatrix},
\qquad 
\begin{pmatrix}
\theta^*_0 &\varphi_1 &  & &{\bf 0}
\\ 
 &\theta^*_1 &\varphi_2
\\
 &  &\theta_2^* &\ddots
 \\
 & & &\ddots &\varphi_{\frac{d}{2}-1}
 \\
{\bf 0}  & & & &\theta^*_{\frac{d}{2}-1}
\end{pmatrix},
$$
respectively, where 
\begin{align*}
\theta_i &= k_0 k_1 q^{2i+2}+k_0^{-1} k_1^{-1} q^{-2i-2}
\qquad 
\hbox{for $i=0,1,\ldots, \textstyle\frac{d}{2}-1$},
\\
\theta^*_i &=
 k_0 k_3 q^{2i+2}+k_0^{-1} k_3^{-1} q^{-2i-2}
\qquad 
\hbox{for $i=0,1,\ldots, \textstyle\frac{d}{2}-1$},
\\
\varphi_i &=
 k_0^{-1} k_2  q^{\frac{d}{2}-1}
 (q^i-q^{-i})
 (q^{i-\frac{d}{2}}-q^{\frac{d}{2}-i})
 \\
& \qquad \times \;
 (q^{-i}-k_0^2 q^{i+2})
 (q^{-i}-k_2^{-2} q^{i-d})
\qquad 
\hbox{for $i=1,2,\ldots, \textstyle\frac{d}{2}-1$}.
\end{align*}
The elements $\alpha,\beta,\gamma$ act on $O(k_0,k_1,k_2,k_3)/O(k_0,k_1,k_2,k_3)(k_0)$ as scalar multiplication by 
\begin{gather}
(k_3+k_3^{-1})(k_2+k_2^{-1})+(k_1+k_1^{-1})(q k_0+q^{-1} k_0^{-1}),
\label{alpha:O/O(k0)}
\\
(k_2+k_2^{-1})(k_1+k_1^{-1})+(k_3+k_3^{-1})(q k_0+q^{-1} k_0^{-1}),
\label{beta:O/O(k0)}
\\
(k_1+k_1^{-1})(k_3+k_3^{-1})+(k_2+k_2^{-1})(q k_0+q^{-1} k_0^{-1}),
\label{gam:O/O(k0)}
\end{gather}
respectively.
\end{lem}
\begin{proof}
By Lemma \ref{lem3:t0_O} the cosets (\ref{basis:O/O(k0)}) are an $\F$-basis for $O(k_0,k_1,k_2,k_3)/O(k_0,k_1,k_2,k_3)(k_0)$. Applying Lemmas \ref{lem:AB_O} and  \ref{lem3:t0_O} a direct calculation yields the matrices representing 
$A$ and $B$ with respect to (\ref{basis:O/O(k0)}). 
By Lemma \ref{lem1:t0_O} we have 
\begin{gather*}
(t_0-k_0^{-1})v_i \in O(k_0,k_1,k_2,k_3)(k_0)
\qquad 
\hbox{for $i=1,3,\ldots,d-1$}.
\end{gather*}
Combined with Theorem \ref{thm:hom} and Proposition \ref{prop:O}(ii) the elements $\alpha,\beta,\gamma$ act on the quotient $\triangle_q$-module of $O(k_0,k_1,k_2,k_3)$ modulo $O(k_0,k_1,k_2,k_3)(k_0)$ as scalar multiplication by (\ref{alpha:O/O(k0)})--(\ref{gam:O/O(k0)}), respectively. 
\end{proof}

\begin{prop}\label{prop:O/O(k0)}
Suppose that $d\geq 2$. Then the $\triangle_q$-module $O(k_0,k_1,k_2,k_3)/O(k_0,k_1,k_2,k_3)(k_0)$ is isomorphic to 
$$
V_{\frac{d}{2}-1}\left(
k_0 k_1 q^{\frac{d}{2}+1},
k_0 k_3 q^{\frac{d}{2}+1},
k_0 k_2 q^{\frac{d}{2}+1}
\right).
$$
Moreover the $\triangle_q$-module 
$
O(k_0,k_1,k_2,k_3)/O(k_0,k_1,k_2,k_3)(k_0)
$ 
is irreducible if the $\H_q$-module $O(k_0,k_1,k_2,k_3)$ is irreducible.
\end{prop}
\begin{proof}
Set $(a,b,c)=(
k_0 k_1 q^{\frac{d}{2}+1},
k_0 k_3 q^{\frac{d}{2}+1},
k_0 k_2 q^{\frac{d}{2}+1})$ and $d'=\frac{d}{2}-1$. Under the assumption (\ref{k0123=-d-1}) the scalar (\ref{beta:O/O(k0)}) is equal to 
$$
(c+c^{-1})(a+a^{-1})+(b+b^{-1})(q^{d'+1}+q^{-d'-1}).
$$
By Proposition \ref{prop:UAWd} along with Lemma \ref{lem:AB_O/O(k0)} the quotient $\triangle_q$-module of $O(k_0,k_1,k_2,k_3)$ modulo $O(k_0,k_1,k_2,k_3)(k_0)$ is isomorphic to $V_{d'}(a,b,c)$. 

Suppose that the $\H_q$-module $O(k_0,k_1,k_2,k_3)$ is irreducible. Using Theorem \ref{thm:irr_O} yields that 
\begin{align*}
&abc,a^{-1}bc,ab^{-1}c,abc^{-1}\not\in\{q^{2i-d'-1}\,|\,i=1,2,\ldots,d'+1\}.
\end{align*}
Hence the $\triangle_q$-module $V_{d'}(a,b,c)$ is irreducible by Theorem \ref{thm:irr_UAW}. The proposition follows.
\end{proof}

For the rest of this subsection we let 
$$
O(k_0,k_1,k_2,k_3)(k_0)'
$$ 
denote the $\F$-subspace of 
$O(k_0,k_1,k_2,k_3)(k_0)$ spanned by
$v_{i-1}+q^{-i}(v_i-v_{i-1})$ for all $i=2,4,\ldots,d$.

\begin{lem}\label{lem:AB_O(k0)'}
Suppose that $d\geq 2$ and $k_0^2=1$. Let
$$
\mu_i=(-1)^{\frac{i}{2}-1} k_0^{1-\frac{i}{2}} k_1^{1-\frac{i}{2}} q^{-\frac{i(i-2)}{4}}
\qquad 
\hbox{for $i=2,4,\ldots,d$}.
$$
Then $O(k_0,k_1,k_2,k_3)(k_0)'$ 
is a $\triangle_q$-module and the matrices representing $A$ and $B$ with respect to the $\F$-basis 
\begin{gather}\label{basis:AB_O(k0)'}
\mu_i(v_{i-1}+q^{-i}(v_i-v_{i-1}))
\qquad 
\hbox{for $i=2,4,\ldots,d$}
\end{gather}  
for $O(k_0,k_1,k_2,k_3)(k_0)'$ are 
$$
\begin{pmatrix}
\theta_0 & & &  &{\bf 0}
\\ 
1 &\theta_1 
\\
&1 &\theta_2
 \\
& &\ddots &\ddots
 \\
{\bf 0} & & &1 &\theta_{\frac{d}{2}-1}
\end{pmatrix},
\qquad 
\begin{pmatrix}
\theta_0^* &\varphi_1 &  & &{\bf 0}
\\ 
 &\theta_1^* &\varphi_2
\\
 &  &\theta_2^* &\ddots
 \\
 & & &\ddots &\varphi_{\frac{d}{2}-1}
 \\
{\bf 0}  & & & &\theta^*_{\frac{d}{2}-1}
\end{pmatrix},
$$
respectively, where 
\begin{align*}
\theta_i
&=
k_0 k_1 q^{2i+2}+k_0^{-1} k_1^{-1} q^{-2i-2}
\qquad 
\hbox{for $i=0,1,\ldots,\textstyle\frac{d}{2}-1$},
\\
\theta_i^*
&=
k_0 k_3 q^{2i+2}+k_0^{-1} k_3^{-1} q^{-2i-2}
\qquad 
\hbox{for $i=0,1,\ldots,\textstyle\frac{d}{2}-1$},
\\
\varphi_i &=
 k_0^{-1} k_2  q^{\frac{d}{2}-1}
 (q^i-q^{-i})
 (q^{i-\frac{d}{2}}-q^{\frac{d}{2}-i})
 \\
& \qquad \times \;
 (q^{-i}-k_0^2 q^{i+2})
 (q^{-i}-k_2^{-2} q^{i-d})
\qquad 
\hbox{for $i=1,2,\ldots, \textstyle\frac{d}{2}-1$}.
\end{align*}
\end{lem}
\begin{proof}
Observe from Lemma \ref{lem:AB_O(k0)} that  $O(k_0,k_1,k_2,k_3)(k_0)'$ is $A$- and $\beta$-invariant and it is also $B$-invariant under the assumption $k_0^2=1$. Therefore $O(k_0,k_1,k_2,k_3)(k_0)'$ is a $\triangle_q$-module by Lemma \ref{lem:gen_UAW}. Moreover the matrices representing $A$ and $B$ with respect to (\ref{basis:AB_O(k0)'})
are as stated. The lemma follows.
\end{proof}

\begin{prop}\label{prop:O(k0)'}
Suppose that $d\geq 2$ and $k_0^2=1$. Then the following hold:
\begin{enumerate}
\item The $\triangle_q$-module $O(k_0,k_1,k_2,k_3)(k_0)'$ is isomorphic to 
$$
V_{\frac{d}{2}-1}\left(
k_0 k_1 q^{\frac{d}{2}+1},
k_0 k_3 q^{\frac{d}{2}+1},
k_0 k_2 q^{\frac{d}{2}+1}
\right).
$$
\item If the $\H_q$-module $O(k_0,k_1,k_2,k_3)$ is irreducible then the $\triangle_q$-module $O(k_0,k_1,k_2,k_3)(k_0)'$ is irreducible.

\item The $\triangle_q$-module $O(k_0,k_1,k_2,k_3)(k_0)/O(k_0,k_1,k_2,k_3)(k_0)'$ is isomorphic to 
$$
V_0\left(
k_0 k_1, k_0 k_3,k_0 k_2
\right).
$$
\end{enumerate}
\end{prop}
\begin{proof}
(i): Set $(a,b,c)=(
k_0 k_1 q^{\frac{d}{2}+1},
k_0 k_3 q^{\frac{d}{2}+1},
k_0 k_2 q^{\frac{d}{2}+1})$ and $d'=\frac{d}{2}-1$. Under the assumptions (\ref{k0123=-d-1}) and $k_0^2=1$ the scalar (\ref{beta:O(k0)}) is equal to 
$$
(c+c^{-1})(a+a^{-1})+(b+b^{-1})(q^{d'+1}+q^{-d'-1}).
$$
By Proposition \ref{prop:UAWd} along with Lemma \ref{lem:AB_O/O(k0)} the quotient $\triangle_q$-module of $O(k_0,k_1,k_2,k_3)$ modulo $O(k_0,k_1,k_2,k_3)(k_0)$ is isomorphic to $V_{d'}(a,b,c)$. 

(ii): By Proposition \ref{prop:O(k0)'}(i) the statement (ii) is immediate from Proposition \ref{prop:O/O(k0)}.

(iii): Set $(a,b,c)=(k_0 k_1, k_0 k_3,k_0 k_2)$.
By Lemma \ref{lem:AB_O(k0)} the elements $A,B,\gamma$ act on the $\triangle_q$-module $O(k_0,k_1,k_2,k_3)(k_0)/O(k_0,k_1,k_2,k_3)(k_0)'$ as scalar multiplication by $a+a^{-1}, b+b^{-1},$ and 
\begin{gather}\label{gamma:O(k0)/O(k0)'}
(k_2+k_2^{-1})(k_1+k_1^{-1})+(k_3+k_3^{-1})(q^{-1} k_0+q k_0^{-1}),
\end{gather}
respectively. Under the assumption $k_0^2=1$ the scalar (\ref{gamma:O(k0)/O(k0)'}) is equal to 
$$
(c+c^{-1})(a+a^{-1})+(b+b^{-1})(q+q^{-1}).
$$
Hence the $\triangle_q$-module $O(k_0,k_1,k_2,k_3)(k_0)/O(k_0,k_1,k_2,k_3)(k_0)'$ is isomorphic to $V_0(a,b,c)$ by Proposition \ref{prop:O}. 
\end{proof}

\begin{thm}\label{thm:O}
Assume that the $\H_q$-module $O(k_0,k_1,k_2,k_3)$ is irreducible. Then the following hold:
\begin{enumerate}
\item If $d=0$ then the $\triangle_q$-module $O(k_0,k_1,k_2,k_3)$ is irreducible.
\item If $d\geq 2$ and $k_0^2=1$ then 
\begin{table}[H]
\begin{tikzpicture}[node distance=1.5cm]
 \node (0)                  {$\{0\}$};
 \node (1)  [above of=0]   {$O(k_0,k_1,k_2,k_3)(k_0)'$};
 \node (E)  [above of=1]   {$O(k_0,k_1,k_2,k_3)(k_0)$};
 \node (V)  [above of=E]   {$O(k_0,k_1,k_2,k_3)$};
 \draw (0)   -- (1);
 \draw (1)   -- (E);
 \draw (E)  -- (V);
\end{tikzpicture}
\end{table}
\noindent is the lattice of $\triangle_q$-submodules of $O(k_0,k_1,k_2,k_3)$.

\item  If $d\geq 2$ and $k_0^2\not=1$ then 
 \begin{table}[H]
\begin{tikzpicture}[node distance=1.5cm]
 \node (0)                  {$\{0\}$};
 \node (1)  [above of=0]   {};
 \node (2)  [right of=1]   {};
 \node (3)  [left of=1]   {};
 \node (E1)  [right of=2]  {$O(k_0,k_1,k_2,k_3)(k_0)$};
 \node (E2)  [left of=3]   {$O(k_0,k_1,k_2,k_3)(k_0^{-1})$};
 \node (V) [above of=1]  {$O(k_0,k_1,k_2,k_3)$};
 \draw (0)   -- (E1);
 \draw (0)   -- (E2);
 \draw (E1)   -- (V);
 \draw (E2)  -- (V);
\end{tikzpicture}
\end{table}
\noindent is the lattice of $\triangle_q$-submodules of $O(k_0,k_1,k_2,k_3)$.

\end{enumerate}
\end{thm}
\begin{proof}
(i): 
If $d=0$ then  $O(k_0,k_1,k_2,k_3)$ is one-dimensional irreducible $\triangle_q$-module.

(ii): Suppose that $d\geq 2$ and $k_0^2=1$. Since the $\triangle_q$-submodule $O(k_0,k_1,k_2,k_3)(k_0)'$ of $O(k_0,k_1,k_2,k_3)(k_0)$ is of codimension $1$, the quotient $\triangle_q$-module $O(k_0,k_1,k_2,k_3)(k_0)$ modulo $O(k_0,k_1,k_2,k_3)(k_0)'$ is irreducible. Combined with Propositions \ref{prop:O/O(k0)} and \ref{prop:O(k0)'} the sequence 
\begin{gather}\label{series_O:k02=1}
\{0\}\subset O(k_0,k_1,k_2,k_3)(k_0)'
\subset O(k_0,k_1,k_2,k_3)(k_0)\subset O(k_0,k_1,k_2,k_3)
\end{gather}
is a composition series for the $\triangle_q$-module $O(k_0,k_1,k_2,k_3)$.

By Proposition \ref{prop:irr_AWmodule_in_t0eigenspace} and Lemma \ref{lem2:t0_O}(ii), every irreducible $\triangle_q$-submodule of $O(k_0,k_1,k_2,k_3)$ is contained in $O(k_0,k_1,k_2,k_3)(k_0)$. 
To prove (ii), it remains to show that $O(k_0,k_1,k_2,k_3)(k_0)'$ is the unique irreducible $\triangle_q$-submodule of $O(k_0,k_1,k_2,k_3)(k_0)$. Suppose on the contrary that $W$ is an irreducible $\triangle_q$-submodule of $O(k_0,k_1,k_2,k_3)(k_0)$ different from $O(k_0,k_1,k_2,k_3)(k_0)'$. By irreducibility, we have  $W\cap O(k_0,k_1,k_2,k_3)(k_0)'=\{0\}$. Hence $W$ is of dimension $1$ and 
\begin{gather}\label{direcsum}
O(k_0,k_1,k_2,k_3)(k_0)=W\oplus O(k_0,k_1,k_2,k_3)(k_0)'.
\end{gather}
Applying Jordan--H\"{o}lder theorem to (\ref{series_O:k02=1}) the $\triangle_q$-module $W$ is isomorphic to $O(k_0,k_1,k_2,k_3)(k_0)'$ if $d=2$, or the $\triangle_q$-module $W$ is isomorphic to  $O(k_0,k_1,k_2,k_3)(k_0)/O(k_0,k_1,k_2,k_3)(k_0)'$.

Suppose that $d=2$ and the $\triangle_q$-module $W$ is isomorphic to $O(k_0,k_1,k_2,k_3)(k_0)'$. By Lemma \ref{lem:AB_O(k0)} 
the eigenvalues of $A$ in $O(k_0,k_1,k_2,k_3)(k_0)$ are 
\begin{align*}
\theta_0 &= k_0 k_1+k_0^{-1} k_1^{-1},
\\
\theta_1 &= k_0 k_1 q^2 +k_0^{-1} k_1^{-1} q^{-2}.
\end{align*}
By Lemma \ref{lem:AB_O(k0)'} 
the eigenvalue of $A$ in $O(k_0,k_1,k_2,k_3)(k_0)'$ is 
$\theta_1$.
Combind with (\ref{direcsum}) this implies that $\theta_0=\theta_1$. Since $q^2\not=1$ and $k_0^2=1$ it follows that 
$k_1^2=q^{-2}$, 
a contradiction to Theorem \ref{thm:irr_O}.

Suppose that the $\triangle_q$-module $W$ is isomorphic to $O(k_0,k_1,k_2,k_3)(k_0)/O(k_0,k_1,k_2,k_3)(k_0)'$. By Proposition \ref{prop:O(k0)'}(iii) the elements $A$ and $B$ act on $W$ as the scalars 
\begin{align*}
\theta_0 &= k_0 k_1+k_0^{-1} k_1^{-1},
\\
\theta_0^* &= k_0 k_3 +k_0^{-1} k_3^{-1},
\end{align*} 
respectively. For $v\in O(k_0,k_1,k_2,k_3)(k_0)$ let $[v]$ denote the coordinate vector of $v$ relative to (\ref{basis:O(k0)}). 
Observe that the $\theta_0$-eigenspace of $A$ in $O(k_0,k_1,k_2,k_3)(k_0)$ is equal to $W$.
Hence $W$ contains a vector $w$ such that the last entry of $[w]$ is $1$. By Lemma \ref{lem:AB_O(k0)} the last entry of $[Bw]$ is 
$$
\theta_{\frac{d}{2}}^*=k_0 k_3 q^d+k_0^{-1} k_3^{-1} q^{-d}.
$$ 
Since $w$ is a $\theta_0^*$-eigenvector of $B$ this implies that $\theta_0^*=\theta_{\frac{d}{2}}^*$. Since $q^d\not=1$ and $k_0^2=1$ it follows that $k_3^2=q^{-d}$, a contradiction to Theorem \ref{thm:irr_O}.
Therefore $O(k_0,k_1,k_2,k_3)(k_0)'$ is the unique irreducible $\triangle_q$-submodule of $O(k_0,k_1,k_2,k_3)(k_0)$. The statement (ii) follows.

(iii): Using the given lemmas and propositions in this subsections, the statement (iii) follows by an argument similar to the proof of Theorem \ref{thm:E0}(ii).
\end{proof}

\section{The summary}\label{s:summary}

We summarize \S\ref{s:lattice_E}--\S\ref{s:lattice_O} as follows:

\begin{thm}\label{thm:BIasRmodule}
Let $V$ denote a finite-dimensional irreducible $\H_q$-module. Given any $\theta\in \F$ let $V(\theta)$ denote the null space of $t_0-\theta$ in $V$. Then the following hold:
\begin{enumerate}
\item Suppose that $t_0$ is not diagonalizable on $V$. Then $t_0$ has a unique eigenvalue $\theta\in \{\pm 1\}$ in $V$. Moreover the following hold:
\begin{enumerate}
\item If the dimension of $V$ is even then the lattice of $\triangle_q$-submodules of $V$ is as follows:
\begin{table}[H]
\begin{tikzpicture}[node distance=1.2cm]
 \node (0)                  {$\{0\}$};
 \node (E)  [above of=0]   {$V(\theta)$};
 \node (V) [above of=E]  {$V$};
 \draw (0)   -- (E);
 \draw (E)  -- (V);
\end{tikzpicture}
\end{table}

\item If the dimension of $V$ is odd then the lattice of $\triangle_q$-submodules of $V$ is as follows:
\begin{table}[H]
\begin{tikzpicture}[node distance=1.2cm]
 \node (0)                  {$\{0\}$};
 \node (W)  [above of=0]   {$V(\theta)'$};
 \node (E)  [above of=W]   {$V(\theta)$};
 \node (V) [above of=E]  {$V$};
 \draw (0)   -- (W);
 \draw (W)   -- (E);
 \draw (E)  -- (V);
\end{tikzpicture}
\end{table}
\noindent Here $V(\theta)'$ is the irreducible $\triangle_q$-submodule of $V(\theta)$ that has codimension $1$.
\end{enumerate}

\item Suppose that $t_0$ is diagonalizable on $V$. Then there are at most two eigenvalues of $t_0$ in $V$. Moreover the following hold:
\begin{enumerate}
\item If $t_0$ has exactly one eigenvalue in $V$ then the $\triangle_q$-module $V$ is irreducible of dimension less than or equal to $2$. 

\item If $t_0$ has exactly two eigenvalues in $V$ then there exists a nonzero scalar $\theta\in \F$ with $\theta\not\in\{\pm 1\}$ such that the lattice of $\triangle_q$-submodules of $V$ is as follows:
\begin{table}[H]
\begin{tikzpicture}[node distance=1.2cm]
 \node (0)                  {$\{0\}$};
 \node (W)  [above of=0]   {};
 \node (E1)  [left of=W]   {$V(\theta^{-1})$};
 \node (E2) [right of=W]  {$V(\theta)$};
 \node (V)  [above of=W]   {$V$};
 \draw (0)   -- (E1);
 \draw (0)   -- (E2);
 \draw (E1)  -- (V);
 \draw (E2)  -- (V);
\end{tikzpicture}
\end{table}
\end{enumerate}
\end{enumerate}
\end{thm}

As byproducts of Theorem \ref{thm:BIasRmodule} we have the following corollaries:

\begin{cor}
Let $V$ denote a finite-dimensional irreducible $\H_q$-module. If $\theta$ is an eigenvalue of $t_0$ in $V$ then either $V=V(\theta)$ or the $\triangle_q$-module $V/V(\theta)$ is irreducible.
\end{cor}

\begin{cor}
For any finite-dimensional irreducible $\H_q$-module $V$, the $\triangle_q$-module $V$ is completely reducible if and only if $t_0$ is diagonalizable on $V$.
\end{cor}

\subsection*{Acknowledgements}

The research is supported by the Ministry of Science and Technology of Taiwan under the project MOST 106-2628-M-008-001-MY4.

\subsection*{Conflict of interest statement}

On behalf of all authors, the corresponding author states that there is no conflict of interest.

\bibliographystyle{amsplain}
\bibliography{MP}

\end{document}